\documentclass[11pt]{amsart}

\usepackage[T1]{fontenc}
\usepackage[latin1]{inputenc}
\usepackage[english]{babel}
\usepackage{amsmath,amssymb,amsthm} 
\usepackage{enumitem}
\usepackage{pgfplots}
\usetikzlibrary{calc,intersections}
\usetikzlibrary{positioning}
\usepackage{hyperref}
\usepackage{bbm}
\usepackage{color}
\usepackage{graphicx}
\newtheorem{theorem}{Theorem}[section]
\newtheorem{lemma}[theorem]{Lemma}
\newtheorem{proposition}[theorem]{Proposition}
\newtheorem{corollary}[theorem]{Corollary}

\theoremstyle{definition}
\newtheorem{definition}[theorem]{Definition}
\newtheorem{question}[theorem]{Question}
\newtheorem{remark}[theorem]{Remark}
\newtheorem{example}[theorem]{Example}
\newtheorem*{remark*}{Remark}
\newtheorem*{ifquestion*}{Inner factor question}

\renewcommand{\Re}{\operatorname{Re}}
\newcommand{\Mult}{\operatorname{Mult}}
\newcommand{\C}{\mathbb{C}}
\newcommand{\N}{\mathbb{N}}
\newcommand{\D}{\mathbb{D}}
\newcommand{\T}{\mathbb{T}}
\newcommand{\R}{\mathbb{R}}
\newcommand{\la}{\langle}
\newcommand{\ra}{\rangle}
\newcommand{\eps}{\varepsilon}
\newcommand{\dist}{\operatorname{dist}}
\newcommand{\ones}{\mathbbm{1}}
\newcommand{\wedi}{s}
\title{Inner factors of Dirichlet space functions }
\author[M. Hartz]{Michael Hartz}
\address{Fachrichtung Mathematik, Universit\"at des Saarlandes, 66123 Saarbr\"ucken, Germany}
\email{hartz@math.uni-sb.de}
\author[S. Richter]{Stefan Richter}
\address{Department of Mathematics, The University of Tennessee, Knoxville, TN 37996}
\email{srichter@utk.edu}
\date{\today}
\thanks{M.H. was partially supported by the Emmy Noether Program of the German Research
Foundation (DFG Grant 466012782)}

\renewcommand{\MR}[1]{}

\pgfplotsset{compat=1.18}
\numberwithin{equation}{section}

\newcommand*{\Stolz}[4]{%
      \coordinate (L) at (#1:1);        
      \def\myrad{0.78}                 
      \coordinate (A) at ({#1+180-#2}:\myrad);
      \coordinate (B) at ({#1+180+#2}:\myrad);
      \path[#3!25, draw=#3, thick]
            (L) -- (A) arc ({#1+180-#2}:{#1+180+#2}:\myrad) -- cycle;
      \fill (L) circle (0.015);
      \node[font=\scriptsize, right] at (L) {$\lambda_{#4}$};
      \coordinate (Mid) at ({#1+180-0.8*#2}:\myrad*0.55);
      \node[font=\scriptsize,#3!80!black] at (Mid) {$Z_{#4}$};
    }


\begin{document}

\begin{abstract}
  Every function in the Dirichlet space on the unit disc has an inner/outer factorization. We study which inner functions occur in this way.
  For Blaschke products, this is the well-known question of which subsets of the disc are zero sets for the Dirichlet space.
  We also consider singular inner factors, and in particular prove an analogue of the Shapiro--Shields theorem in this setting. Our results on singular inner factors also yield new sufficient conditions for zero sets.
\end{abstract}

\maketitle

\section{Introduction}

The Dirichlet space $\mathcal{D}$ is the space of analytic functions $f$ on the open unit disc $\mathbb{D}$ such that
\begin{equation*}
  \int_{\mathbb{D}} |f'|^2 \, d A < \infty,
\end{equation*}
where $A$ denotes the normalized area measure on $\mathbb{D}$. Since $\mathcal{D}$ is contained
in the Hardy space $H^2$, every non-zero function $f \in \mathcal{D}$ admits an inner/outer factorization $f = f_i f_o$. Whereas $f_i$ only belongs to $\mathcal{D}$ if it is a finite Blaschke product, the outer factor $f_o$ is always in $\mathcal{D}$, see Section 3.1 and the references provided there. Let us call an inner function $I$ an \emph{inner factor} of $\mathcal{D}$ if
it occurs as the inner factor of some function in $\mathcal{D}$, i.e. if $IH^2\cap \mathcal D\ne (0)$. Equivalently, there exists a non-zero function $g \in \mathcal{D}$ such that $I g \in \mathcal{D}$; see Proposition \ref{prop:inner_factor_char} below.
The basic question we study is the following.
\begin{ifquestion*}
  Which inner functions are inner factors of $\mathcal{D}$?
\end{ifquestion*}

Except for the possibility of a constant factor of modulus one, every inner function is a product of a Blaschke product
\begin{equation*}
  B(z) = z^k \prod_{n} \frac{|z_n|}{z_n} \frac{z_n - z}{1 - \overline{z_n} z},
\end{equation*}
where $k \in \mathbb{N}_0$ and $\sum_{n} (1 - |z_n|) < \infty$, and
a singular inner function
\begin{equation*}
  S_\mu(z) = \exp \Big( - \int_{\mathbb{T}} \frac{w + z}{w - z} \, d \mu(w) \Big),
\end{equation*}
where $\mu$ is a positive measure on the unit circle $\mathbb{T}$ that is singular with respect to Lebesgue measure on $\mathbb{T}$.
It is not hard to show that $B S_\mu$ is a factor of $\mathcal{D}$ if and only if both $B$ and $S_\mu$ are factors; see Corollary \ref{cor:inner_factor_product} below. Thus, the inner factor question can be divided into these two cases.
The question of which Blaschke products are inner factors of $\mathcal{D}$ is the question
of which subsets of the disc are zero sets of functions in $\mathcal{D}$.
This problem has been studied extensively, see e.g.\ \cite[Chapter 4]{EKM+14}.

In this article, we will study the inner factor question also for singular inner functions, in particular
in the case when the measure $\mu$ is atomic.
It turns out that results about singular inner factors imply results about zero sets.
The basic idea is that Carleson's formula allows us to bound the Dirichlet integral of $B g$,
where $B$ is a Blaschke product, in terms of the Dirichlet integral of $S g$,
where $S$ is a singular inner function whose measure is atomic and related to the zeros of $B$ in an appropriate way; see
Theorem \ref{thm:atomic-zero} for a precise statement.
We will use this idea to establish new sufficient conditions for zero sets in $\mathcal{D}$.
We do not know if one can conversely deduce results about singular inner functions from
results about zero sets; see Question \ref{ques:SingVsZeros} for a more precise formulation of this question\footnote{The answer was found during revisions; see Appendix \ref{sec:appendixb}}.
Nonetheless, many of our results can be regarded as evidence in favor of such a general equivalence
between singular inner factors and zero sets.

We now discuss some of our main results. Regarding zero sets, we will prove the following result, which essentially allows us to reduce to the case of sequences in the disc that converge to a single point on the boundary $\mathbb{T}$ of $\mathbb{D}$.

\begin{theorem}
  \label{thm:one_point_intro}
  If $S \subset \mathbb{D}$ is a set without accumulation points in $\mathbb{D}$ that is not a zero set for $\mathcal{D}$, then there exists a sequence $\{z_n\}$ in $S$ that converges to a point on $\mathbb{T}$ such that $\{z_n\}$ is not a zero set for $\mathcal{D}$.
\end{theorem}

In fact, we will show that such a result holds in many Banach spaces of analytic functions on the disc, see Theorem \ref{thm:zero_subset}.
We note that Theorem \ref{thm:one_point_intro} is most relevant in the case when $S$ is not a Blaschke sequence, the other case being quite elementary.

Our next result regarding zero sets allows us to push points out towards the boundary.

\begin{theorem}
  \label{thm:pushing_intro}
  Let $r_n \in (0,1)$, $t_n \in \mathbb{R}$ and assume that $\{r_n e^{ i t_n}: n \in \mathbb{N} \}$ is a zero set for $\mathcal{D}$. Let $s_n \in [r_n,1)$. Then $\{s_n  e^{ i t_n}: n \in \mathbb{N} \}$ is also a zero set for $\mathcal{D}$.
\end{theorem}

We will show a quantitative sharpening of this theorem, and in the appendix, we will exhibit a circularly invariant complete Pick space on the disc in which the quantitative sharpening does not hold.

Broadly speaking, many of the known sufficient conditions for zero sets in $\mathcal{D}$ can be divided into two classes: those that come from rapid convergence of the moduli to $1$ and those that come from a restriction of the distribution of the arguments.
A famous result of Shapiro and Shields \cite{SS62} (see also \cite[Theorem 4.2.1]{EKM+14}) belongs to the first class. It says that if the sequence $\{z_n\}$ satisfies $\sum_n \frac{1}{\log \frac{1}{1-|z_n|}}< \infty$, then $\{z_n\}$ is a zero set for $\mathcal{D}$. Moreover, Nagel--Rudin--Shapiro showed that this condition is best possible among conditions that only depend on the modulus; see \cite{NRS82} or \cite[Theorem 4.2.2]{EKM+14}. We will establish the following analogue of these results for singular inner factors.

\begin{theorem}
  \label{thm:SS_intro}
  The following are equivalent for a summable sequence $\{a_n\}$ of positive numbers:
  \begin{enumerate}[label=\normalfont{(\roman*)}]
    \item
  for any sequence $\{\lambda_n\}$ in $\mathbb{T}$, the singular inner function $S_\mu$
  associated with the measure $\mu = \sum_n a_n \delta_{\lambda_n}$ is a factor of $\mathcal{D}$;
  \item for any sequence $\{\lambda_n\}$ in $\mathbb{T}$ with $\lambda_n \to 1$, the singular inner function $S_\mu$
  associated with the measure $\mu = \sum_n a_n \delta_{\lambda_n}$ is a factor of $\mathcal{D}$;
\item $\sum_n \frac{1}{\log ( 1+  \frac{1}{a_n})}< \infty$.
  \end{enumerate}
\end{theorem}

From this result, we obtain sufficient conditions for zero sets in $\mathcal{D}$.
Here, we mention one.
If $\lambda \in \mathbb{T}$ and $K > 1$, let
\begin{equation*}
  \Omega_K(\lambda) = \{z \in \mathbb{D}: |\lambda - z| \le K(1 - |z|) \}
\end{equation*}
be the Stolz angle with vertex $\lambda$ and aperture $K$.

\begin{theorem}
  \label{thm:zero_decomp_intro}
  Let $Z$ be a sequence in $\mathbb{D}$. Suppose that $Z$ can be partitioned as $Z = \bigcup_{k} Z_k$, such that
  \begin{itemize}
    \item there exist $\lambda_k \in \mathbb{T}$ and $K(k) > 1$ such that $Z_k \subset \Omega_{K(k)}(\lambda_k)$, and
    \item setting $a_k = 2 K(k)^2 \sum_{z \in Z_k} (1 - |z|)$, we have
      \begin{equation}
        \label{eqn:zero_boundary_hyp}
        \sum_{k} \frac{1}{\log(1 + \frac{1}{a_k})} < \infty.
      \end{equation}
  \end{itemize}
  Then $Z$ is a zero set for the Dirichlet space.
\end{theorem}

\begin{figure}[ht]
  \begin{tikzpicture}[scale=2.5]
    \draw[thick] (0,0) circle (1);


      \Stolz{  0}{60}{red}{1}
      \Stolz{ 20}{30}{blue}{2}
      \Stolz{ 40}{15}{purple}{3}

    \foreach \ang/\radpt in {0/0.55,10/0.65,-10/0.65,5/0.70,-5/0.70}{%
      \fill[red]   ({\ang}:{\radpt}) circle (0.013);
    }

    \foreach \ang/\radpt in {20/0.70,22/0.71,18/0.69,24/0.70,16/0.70}{%
      \fill[blue]  ({\ang}:{\radpt}) circle (0.013);
    }

    \foreach \ang/\radpt in {40/0.60,38/0.70,42/0.70}{%
      \fill[purple] ({\ang}:{\radpt}) circle (0.013);
    }
  \end{tikzpicture}
    \caption{The setting of Theorem \ref{thm:zero_decomp_intro}}
    \label{fig:Stolz}
\end{figure}
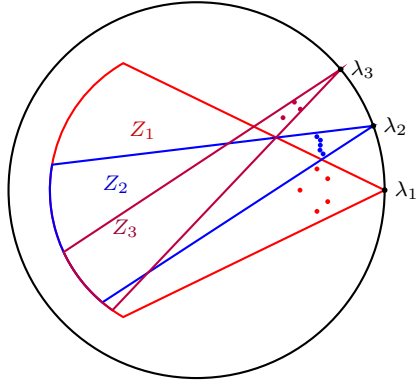

At one extreme, if we take each $Z_k$ to be a singleton $\{z_k\}$, then by choosing $K(k)$ such that $a_k=3(1-|z_k|)$  and $\lambda_k=z_k/|z_k|$ we recover the Shapiro--Shields theorem. At the other extreme, if $Z_1 = Z$ and all other $Z_k$ are empty, then we recover the known  fact that every Blaschke sequence contained in a Stolz angle is a zero set for $\mathcal{D}$.

We turn to sufficient conditions that come from the distribution of the arguments.
Recall that a compact set $E \subset \mathbb{T}$ is a Carleson set if
\begin{equation*}
  \int_{\mathbb{T}} \log\Big( \frac{1}{\dist(z,E)} \Big) \, |dz| < \infty.
\end{equation*}
Let $A^\infty$ denote the algebra of all analytic functions on $\mathbb{D}$ all of whose derivatives are bounded. Clearly, $A^\infty \subset \mathcal{D}$.
Zero sets for $A^\infty$ have been completely characterized by Caughran \cite{Caughran70},  Taylor--Williams \cite{TW70} and Nelson \cite{Nelson71}. For ease of reference, we state a combination of their results.

\begin{theorem}[Caughran, Taylor--Williams, Nelson]
  \label{thm:CTWN}
  Let $Z = \{z_n\}$ be a sequence in $\mathbb{D} \setminus \{0\}$. The following are equivalent:
  \begin{enumerate}[label=\normalfont{(\roman*)}]
    \item $Z$ is a zero set for $A^\infty$;
    \item $Z$ is Blaschke and the set $\overline{\{\frac{z_n}{|z_n|}\}}$ is a Carleson set;
    \item $Z$ is Blaschke and $\int_{\mathbb{T}} \log \Big( \frac{1}{\dist(z,Z)} \Big) \, |dz| < \infty$.
  \end{enumerate}
\end{theorem}

In particular, either (ii) or (iii) is a sufficient condition for $Z$ to be a zero set for the Dirichlet space. In Theorem \ref{thm:rate_of_convergence} we observe an obstacle to expanding on these conditions: any sufficient condition for zero sets in the Dirichlet space that only depends on the rate of convergence of the sequence to a single boundary point is necessarily implied by Theorem \ref{thm:CTWN}.

The analogue of Theorem \ref{thm:CTWN} for singular inner functions was shown by Taylor and Williams;
see \cite[Corollary 4.8]{TW70}.

\begin{theorem}[Taylor--Williams]
  \label{thm:TW_singular}
  A singular inner function $S_\mu$
  is a factor of $A^\infty$ if and only if the support of $\mu$ is a Carleson set.
\end{theorem}

Related to Theorem \ref{thm:CTWN}, Bogdan \cite{Bogdan96} proved that a set $S \subset \mathbb{D}$
has the property that every Blaschke sequence in $S$ is a zero set for the Dirichlet space if and only if $\int_{\mathbb{T}} \log \Big( \frac{1}{\dist(z,S)} \Big) \, |dz| < \infty$. Moreover,
Caughran proved that a compact set $E \subset \mathbb{T}$ has the property
that every Blaschke sequence in $\mathbb{D}$ for which $\{\frac{z_n}{|z_n|}\} \subset E$ is a zero set for the Dirichlet space if and only if $E$ is a Carleson set; see \cite[Theorem 2]{Caughran70} and its proof or \cite[Section 4.5]{EKM+14}.
We will establish the following boundary analogue of these results.

\begin{proposition}
  \label{prop:Carleson_singular_intro}
  The following are equivalent for a compact set $E \subset \mathbb{T}$:
  \begin{enumerate}[label=\normalfont{(\roman*)}]
    \item Every singular inner function associated with a measure supported on $E$ is a factor of $\mathcal{D}$;
    \item $E$ is a Carleson set.
  \end{enumerate}
\end{proposition}

We now have two sufficient conditions for singular inner factors, the Shapiro--Shields type condition in Theorem \ref{thm:SS_intro} and the Carleson set condition in Proposition \ref{prop:Carleson_singular_intro}. We will show that there are singular inner factors that are not covered by combining these two conditions; see Theorem \ref{thm:new_singular_inner_factor} for the precise statement.
As a consequence, there are zero sets for the Dirichlet space that are not the union of a sequence satisfying the Shapiro--Shields condition and a sequence satisfying the Carleson set condition of Theorem \ref{thm:CTWN}.

Finally, in Theorem \ref{thm:SS_wedi}, we will prove an analogue of the sufficiency part of Theorem \ref{thm:SS_intro} in standard weighted Dirichlet spaces.

The paper is structured as follows. We start with a section with general observations about zero sets and extremal functions in spaces of analytic functions. That section ends with a version of Theorem \ref{thm:one_point_intro} that is valid for many Hilbert function spaces with complete Pick kernels. This includes the standard weighted Dirichlet spaces that are contained in $H^2$. In Section \ref{sec:prelim} we have listed known facts about the Dirichlet space along with some corollaries that will be used repeatedly. Our main theorems about the Dirichlet space are contained in Sections 4-9. Section 10 contains Theorem \ref{thm:SS_wedi} about atomic singular inner factors in weighted Dirichlet spaces. Appendix \ref{sec:Appendix A} contains the construction of a Pick space in which our result about pushing out points fails.

\section{Uniqueness sets of Banach spaces of analytic functions}
We say $\mathcal B$ is a Banach  space of analytic functions on $\mathbb{D}$, if $\mathcal B$ is a Banach space that is injectively embedded in $ \mathrm{Hol}(\D)$, and for every $z\in \D$ the point evaluation at $z$ is a continuous functional on $\mathcal B$.

\begin{definition} \label{def:Fatou Prop}
We say that $\mathcal{B}$ has the \emph{Fatou property} if for every bounded sequence $\{f_n\}$ in $\mathcal{B}$ that converges locally uniformly on $\mathbb{D}$ to a holomorphic function $f$, we have $f\in \mathcal{B}$ and $\|f\|_{\mathcal{B}}\le \liminf_{n\to \infty} \|f_n\|_{\mathcal{B}}$.\end{definition}
Note that by the uniform boundedness principle, every bounded subset of $\mathcal{B}$ is a normal family; hence we could replace the assumption of local uniform convergence by the assumption of pointwise convergence in the definition of the Fatou property.

\begin{example}
  \label{exa:Fatou_basic}
  \begin{enumerate}[label=\normalfont{(\alph*)},wide]
    \item The Dirichlet space $\mathcal{D}$ and the Hardy spaces $H^p$ for $1 \le p \le \infty$ have the Fatou property (by Fatou's lemma). More generally, this argument applies to Besov--Sobolev spaces of holomorphic functions on $\mathbb{D}$.
    \item Every dual Banach space of analytic functions in which point evaluations are weak-$*$ continuous has the Fatou property by Alaoglu's theorem.
    \item The disc algebra $A(\mathbb{D})$ does not have the Fatou property. For instance,
      if $f \in H^\infty \setminus A( \mathbb{D})$ and $\{r_n\}$ is a sequence in $(0,1)$ that converges to $1$, then the sequence $\{f_{r_n}\}$ witnesses the failure of the Fatou property.
  \end{enumerate}
\end{example}

The Fatou property implies the following basic compactness principle.

\begin{lemma}
  \label{lem:compact}
  Let $\mathcal{B}$ be a Banach space of analytic functions on $\mathbb{D}$ that has the Fatou property.
  Let $\{f_n\}$ be a bounded sequence in $\mathcal{B}$.
  Then $\{f_n\}$ has a subsequence that converges locally uniformly to a function $f \in \mathcal{B}$, and $\|f\| \le \liminf_{n \to \infty} \|f_n\|$.
\end{lemma}

\begin{proof}
  As explained above, $\{f_n\}$ is a normal family, so by Montel's theorem there is a subsequence that converges locally uniformly to a holomorphic function $f$.
  By the Fatou property, we have $f \in \mathcal{B}$ and $\|f\| \le \liminf_{n \to \infty} \|f_n\|$.
\end{proof}

If $\mathcal{B}$ is a Banach space of analytic functions on $\mathbb{D}$ and $\mathcal{M} \subset \mathcal{B}$ is a subspace, we define
\begin{equation*}
  \gamma(\mathcal M)=\gamma_{\mathcal B}(\mathcal M)=\sup\{ \Re f(0): f\in \mathcal M, \|f\|_{\mathcal{B}}\le 1\}.
\end{equation*}
If $\mathcal{M}$ has the Fatou property (e.g. if $\mathcal{M}$ is a closed subspace of the Dirichlet space), then the supremum is attained by Lemma \ref{lem:compact}. Moreover, if the norm of $\mathcal{B}$ is strictly convex (as is the case in any Hilbert space) and if $\gamma(\mathcal M)>0$, then the supremum is attained by a unique function, which we call the extremal function of $\mathcal{M}$.

We will make use of the following continuity property of the function $\gamma$.

\begin{lemma}
  \label{lem:gamma_continuous}
  Let $\mathcal{B}$ be a Banach space of analytic functions on $\mathbb{D}$
  and let $\{\mathcal{M}_n\}$ be a decreasing sequence of closed subspaces of $\mathcal{B}$
  that all have the Fatou property.
  Then
  \begin{equation*}
    \gamma\Big(\bigcap_n \mathcal{M}_n\Big) = \lim_{n \to \infty} \gamma(\mathcal{M}_n).
  \end{equation*}
\end{lemma}

\begin{proof}
  It is obvious that $\gamma\Big(\bigcap_n \mathcal{M}_n\Big) \le \lim_{n \to \infty} \gamma(\mathcal{M}_n)$, and the limit exists because the sequence is decreasing.
  Conversely, for each $n$, Lemma \ref{lem:compact} implies that there is a function $f_n \in \mathcal{M}_n$ with $\|f_n\|_{\mathcal{B}} \le 1$ and $\Re f_n(0) = \gamma(\mathcal{M}_n)$.
  Applying Lemma \ref{lem:compact} again, we find a subsequence of $\{f_n\}$ that converges locally uniformly to a holomorphic function $f$. By the Fatou property, we have $f \in \bigcap_n \mathcal{M}_n$ and $\|f\|_{\mathcal{B}} \le 1$. Hence
  \begin{equation*}
    \gamma\Big(\bigcap_n \mathcal{M}_n\Big) \ge \Re f(0) = \lim_{n \to \infty} \gamma(\mathcal{M}_n). \qedhere
  \end{equation*}

\end{proof}

Let $S=\{z_n\}$ be a sequence of points in $\D$ (finitely many repetitions allowed, and the sequence could possibly be finite).
We say that a holomorphic function $f$ vanishes on $S$ if $f$ has a zero of multiplicity $\ge k$ at $z$ for each $z$ that is repeated $k$ times in the sequence $S$.
We say that $f$ vanishes exactly on $S$ if $f$ has a zero of multiplicity $k$ at $z$ for each $z$ that is repeated $k$ times in the sequence $S$, and $f(z)\ne 0$ for every $z\in \D\setminus S$.

We say that $S$ is a zero set for $\mathcal B$, if there is a function $f\in \mathcal B$
that vanishes exactly on $S$.
We say that $S$ is a set of uniqueness for $\mathcal B$, if $f=0$ is the only function $f\in \mathcal B$ that vanishes on $S$.
If $\mathcal B$ has the property that subsets of zero sets are zero sets, then a sequence is either a zero set or a set of uniqueness. Furthermore, under the same hypothesis a sequence
$S$ in $\mathbb{D} \setminus \{0\}$ is a uniqueness set if and only if every function $f\in \mathcal B$ that vanishes on $S$ is zero at the origin.

We can now prove a more general version of Theorem \ref{thm:one_point_intro}.

\begin{theorem}
  \label{thm:zero_subset}
  Let $\mathcal{B}$ be a Banach  space of analytic functions on $\D$ such that
\begin{itemize}
  \item $\mathcal{B}$ has the Fatou property,
  \item  subsets of zero sets are zero sets,
\item finite unions of zero sets are zero sets, and
\item finite sets are zero sets.
\end{itemize}
If $S$ is a sequence in $\mathbb{D}$ without accumulation points in $\D$ that is a set of uniqueness for $\mathcal{B}$, then there exists $w\in \T$ and a subsequence $\{z_n\}$ of $S$ such that $z_n\to w$ and $\{z_n\}$ is a set of uniqueness for $\mathcal B$.
\end{theorem}

This result applies in particular to the Dirichlet space: the Fatou property was mentioned in Example \ref{exa:Fatou_basic}, the subset property is a consequence of Carleson's formula (see Section \ref{sec:prelim} or e.g. \cite[Theorem 4.1.4]{EKM+14}), the union property follows from the fact that the zero sets of $\mathcal{D}$ and of $\Mult(\mathcal{D})$ agree (see e.g. \cite[Corollary 5.4.2]{EKM+14} or \cite{AHM+17a}),
and finite sets are zero sets since $\mathcal{D}$ contains the polynomials.

\begin{proof}
  If $\mathbb{T}$ is covered by finitely many open discs $D_1,\ldots,D_n$,
  then all but finitely many points of $S$ lie in $\bigcup_{j=1}^n D_j$.
  Hence the hypotheses imply that one of the sets $S\cap D_j$ is a set of uniqueness for $\mathcal{B}$. Thus, by compactness of $\mathbb{T}$, there is $w\in \T$ such that $S\cap B_\varepsilon(w)$ is a set of uniqueness for $\mathcal{B}$ for every $\eps>0$.

  If $F \subset S$, we write
  \begin{equation*}
    \mathcal{M}_F = \{ f \in \mathcal{B}: f \text{ vanishes on } F\}
  \end{equation*}
  and $\gamma(F) = \gamma(\mathcal{M}_F)$. Notice that since $\mathcal B$ has the Fatou property, it follows that each $\mathcal M_F$ will have the Fatou property as well. Thus, if  $\varepsilon,\delta > 0$, then since $S \cap B_\varepsilon(w)$ is a uniqueness set,
  Lemma \ref{lem:gamma_continuous} implies that there exists
  a finite subset $F \subset S$ such that $\gamma(F \cap B_\varepsilon(w)) < \delta$.
  Hence there is a sequence of finite subsets $F_n \subset S \cap B_{\frac{1}{n}}(w)$ such that $\gamma(F_n) < \frac{1}{n}$.
  Let $F = \bigcup_{n=1}^\infty F_n$, where we use the convention
  that points are only repeated if they appear with repetition in the same $F_n$. 
  Then   $F$ is a subsequence of $S$ that converges to $w$.
  Since $\mathcal M_F \subset \mathcal M_{F_n}$ we have $\gamma(F) \le \gamma(F_n) < \frac{1}{n}$ for each $n$. Thus, $\gamma(F) = 0$,
  hence as noted above, $F$ is a set of uniqueness.
\end{proof}

\begin{remark}
  Note that the conclusion of Theorem \ref{thm:zero_subset} is false for the disc algebra
  (which does not satisfy the Fatou property). Indeed, the zero sets of the disc algebra are precisely the Blaschke sequences whose cluster set has linear Lebesgue measure zero; see for instance \cite[Chapter 6]{Hoffman62}.
  So if $Z$ is any Blaschke sequence whose cluster set has positive Lebesgue measure, then $Z$ is a set of uniqueness for the disc algebra, but any subsequence of $Z$ that converges to a single point on the circle is a zero set for the disc algebra.
\end{remark}

\begin{remark}
  Nagel, Rudin and Shapiro \cite{NRS82} (see also \cite[Theorem 4.2.2]{EKM+14}) showed that if $\sum_{n} \frac{1}{\log(\frac{1}{1-r_n})} = \infty$,
  then there exists a sequence $\{z_n\}$ with $|z_n| = r_n$ such that $\{z_n\}$ is a set of uniqueness
  for $\mathcal{D}$.
  By Theorem \ref{thm:zero_subset}, we see that there exists a subsequence $(z_{n_k})$ that is a set of uniqueness and that converges to a point on the boundary.
  Moreover, by the theorem of Shapiro and Shields, we must have $\sum_{k} \frac{1}{\log( \frac{1}{1 - |z_{n_k}|})} = \infty$. This gives another argument for the result of \cite{RRS04}. Similarly, for weighted Dirichlet spaces the above theorem can be applied to give a new proof of Theorem 4 of \cite{PauPe11}.
\end{remark}

\begin{question}
  Does the conclusion of Theorem \ref{thm:zero_subset} hold for the Bergman space?
\end{question}

Since unions of zero sets in the Bergman space need not be zero sets (see e.g.\ \cite[Theorem 4.3]{DS04}), Theorem \ref{thm:zero_subset} does not apply directly.

\section{The Dirichlet space: Preliminaries}
\label{sec:prelim}

In this section, we collect a few  results about the Dirichlet space, which we will use throughout.
Our standard reference is the textbook \cite{EKM+14}.
Unless otherwise specified, we will equip $\mathcal{D}$ with the norm
\begin{equation} 
  \label{eqn:Dirichlet_norm}
  \|f\|^2 = \int_{\mathbb{D}} |f'|^2 d A + \|f\|_{H^2}^2 = \sum_{n=0}^\infty (n+1) |\widehat{f}(n)|^2
\end{equation}
where $A$ denotes the normalized area measure,  $H^2$ is the classical Hardy space on the disc, and $f(z) = \sum_{n=0}^\infty \widehat{f}(n) z^n$ is the Taylor expansion of $f$ at the origin.
The multiplier algebra of $\mathcal{D}$ is defined to be
\[
  \Mult(\mathcal{D}) = \{\varphi: \mathbb{D} \to \mathbb{C}: \varphi \cdot f \in \mathcal{D} \text{ for all } f \in \mathcal{D} \}.
\]
Equipped with the multiplier norm
\[
  \|\varphi\|_{\Mult(\mathcal{D})} = \sup \{ \|\varphi f \|: f \in \mathcal{D}, \|f\| \le 1 \},
\]
the multiplier algebra is a unital commutative Banach algebra.
It is elementary that $\Mult(\mathcal{D}) \subset \mathcal{D} \cap H^\infty$,
and it is also known that the inclusion is proper;
see \cite[Section 5.1]{EKM+14}.

The Dirichlet space is a reproducing kernel Hilbert space (RKHS) on $\mathbb{D}$ with reproducing kernel
\begin{equation*}
  K(z,w) = \frac{1}{z \overline{w}} \log\Big( \frac{1}{1 - z \overline{w}} \Big).
\end{equation*}
Occasionally we will reference results that hold for the Dirichlet space, because $K(z,w)$ is a complete Pick kernel; see \cite{AM02} for definitions and further background. One such fact has already been mentioned,
namely that finite unions of zero sets are zero sets.
Another very useful consequence of the complete Pick property is that
the extremal functions of multiplier invariant subspaces are contractive multipliers;  see \cite{MT2000} or \cite[Corollary 4.2]{AHM+18}, or see \cite[Theorem 5.1]{RS94}, where different methods were used and this fact was established only for the Dirichlet space.

\subsection{Carleson's formula}
As mentioned in the introduction, functions in $\mathcal{D}$ have an inner/outer factorization.
Carleson's formula allows us to compute the norm of a function in terms of its inner/outer factorization.
To start, the local Dirichlet integral of a function $f \in H^2$ at $z \in \mathbb{T}$ is defined to be
\begin{equation*}
  D_z(f) = \int_{\mathbb{D}} |f'(w)|^2 \frac{1 - |w|^2}{|w - z|^2} \, d A(w).
\end{equation*}
If $f \in H^2$ and $I$ is an inner function, then
\begin{equation} \label{equ:D_z(I)}
  \|I f\|^2 = \int_{\mathbb{T}} D_z(I) |f(z)|^2 \frac{|dz|}{2 \pi} + \|f\|^2;
\end{equation}
see for instance Theorems 7.1.3 and 7.6.1 in \cite{EKM+14}.
In particular, if $I f \in \mathcal{D}$, then $f \in \mathcal{D}$.
Moreover,
if $I = B S_\mu$, where $B$ is the Blaschke product with zeros $\{z_n\}$, then
\begin{equation}
  \label{eqn:local_Diri_inner}
  D_z(I) = D_z(B) + D_z(S_\mu) = \sum_n \frac{1 - |z_n|^2}{|z - z_n|^2} + \int_\mathbb{T} \frac{2}{|z - w|^2} d \mu(w);
\end{equation}
see Theorem 7.6.7 in \cite{EKM+14}.

We record two basic observations, part of which were already mentioned in the introduction.
\begin{proposition}
  \label{prop:inner_factor_char}
  The following are equivalent for an inner function $I$:
  \begin{enumerate}[label=\normalfont{(\roman*)}]
    \item $I$ occurs as the inner factor of some function in $\mathcal{D}$;
    \item there exists $g \in H^2 \setminus \{0\}$ such that $I g \in \mathcal{D}$;
    \item there exists $\varphi \in \Mult(\mathcal{D}) \setminus \{0\}$ such that $I \varphi \in \Mult(\mathcal{D})$.
  \end{enumerate}
\end{proposition}

\begin{proof}
  (i) $\Rightarrow$ (ii) is trivial.

  (ii) $\Rightarrow$ (i) Let $g \in H^2 \setminus \{0\}$ such that $I g \in \mathcal{D}$ and let $g = g_i g_o$ be the inner/outer factorization of $g$. Then $g_i I g_o \in \mathcal{D}$, so $I g_o \in \mathcal{D}$ by Carleson's formula, and $I$ is the inner factor of $I g_o$.

  (ii) $\Rightarrow$ (iii) Let $g \in H^2 \setminus \{0\}$ such that $I g \in \mathcal{D}$.
  By \cite[Theorem 1.1]{AHM+17a}, there exist $\varphi_1,\varphi_2 \in \Mult(\mathcal{D})$ with $\varphi_2$ non-vanishing such that
  $I g = \frac{\varphi_1}{\varphi_2}$.
  Define $\varphi = \varphi_2 g$, which is not zero. Then $I \varphi = \varphi_1 \in \Mult(\mathcal{D})$.
  Hence for all $h \in \mathcal{D}$, we have $I \varphi h = \varphi_1 h \in \mathcal{D}$, and so $\varphi h \in \mathcal{D}$
  by Carleson's formula. This implies that $I\varphi \in \Mult(\mathcal{D})$ and $\varphi \in \Mult(\mathcal{D})$.

  (iii) $\Rightarrow$ (ii) is trivial.
\end{proof}

\begin{corollary}
  \label{cor:inner_factor_product}
  Let $I_1,I_2$ be two inner functions. Then $I_1 I_2$ is an inner factor of $\mathcal{D}$ if and only if both $I_1$ and $I_2$ are inner factors of $\mathcal{D}$.
\end{corollary}

\begin{proof}
  The ``only if'' statement follows from Carleson's formula since $D_z(I_1I_2)=D_z(I_1)+D_z(I_2)$ by (\ref{eqn:local_Diri_inner}), while ``if'' follows from (iii)
  and the observation that $\Mult(\mathcal{D})$ is an algebra without non-trivial zero divisors.
\end{proof}

We will make use of the following standard application of infinite products.
\begin{lemma}
  \label{lem:infinite_products}
  Let $I$ be an inner function that can be written as an infinite product $I = \prod_{n} I_n$, where each $I_n$
  is an inner function with $I_n(0) > 0$.
  Assume that each $I_n$ is a factor of $\mathcal{D}$ and that
  \begin{equation*}
    \sum_{n} (1  - \gamma(I_n H^2 \cap \mathcal{D})) < \infty.
  \end{equation*}
  Then $I$ is an inner factor of $\mathcal{D}$.
\end{lemma}

\begin{proof}
  Let $\varphi_n$ be the extremal function for $I_n H^2 \cap \mathcal{D}$. Then each $\varphi_n$ is a contractive multiplier; see the discussion at the beginning of the section. We claim that the infinite product
  \begin{equation*}
    \varphi = \prod_{n=1}^\infty \varphi_n
  \end{equation*}
  converges locally uniformly to a nonzero function $\varphi \in \Mult(\mathcal{D})$,
  and that $\varphi = I \psi$ for some $\psi \in H^2$.

  To see this, we use the following basic result from complex analysis: If $\{f_n\}$ is a sequence of holomorphic functions on $\mathbb{D}$ bounded by $1$ such that $\sum_n |1 - f_n(0)| < \infty$, then the infinite product $\prod_{n=1}^\infty f_n$ converges absolutely and locally uniformly  to a bounded holomorphic function.
  Indeed, a routine estimate based on the Schwarz--Pick lemma shows that if $|z| \le r$, then
  \begin{equation*}
    |1 - f_n(z)| \le \frac{1 + r}{1 -r } |1 - f_n(0)|,
  \end{equation*}
  and hence convergence follows from the standard theory of infinite products; see e.g. \cite[VII.5]{Conway78}.

  Since $\varphi_n(0) = \gamma(I_n H^2 \cap \mathcal{D})$, the assumption implies that the product defining $\varphi$ converges uniformly on compact subsets of $\mathbb{D}$. Since the partial products are bounded in multiplier norm,
  the Fatou property of $\mathcal{D}$ implies that $\varphi \in \Mult(\mathcal{D})$.
  (Indeed, the infinite product converges in the weak-$*$ topology of $\Mult(\mathcal{D})$,
  obtained by viewing $\Mult(\mathcal{D})$ as a weak-$*$ closed subalgebra of the algebra of all bounded linear operators on $\mathcal{D}$.)
  Writing $\varphi_n = I_n \psi_n$, we also have $\|\psi_n\|_\infty \le 1$ and $\psi_n(0) \ge \varphi_n(0)$, so the infinite product
  \begin{equation*}
    \psi = \prod_{n=1}^\infty \psi_n
  \end{equation*}
  converges absolutely to a bounded holomorphic function.
  Since $I_n(0) \ge \varphi_n(0)$, the product $I = \prod_{n=1}^\infty I_n$ also converges absolutely,
  and hence $\varphi = I \psi$.
\end{proof}

\subsection{Atomic singular inner factors and zero sets}
We next record another consequence of Carleson's formula that we will use repeatedly.
We start with a lemma, which is probably well known. We include the short proof for completeness.

\begin{lemma}
  \label{lem:Poisson_sup}
  Let $z\in \D$ and $\lambda \in \mathbb{T}$.
  Then
  \begin{equation*}
     \sup_{|\zeta|= 1} \frac{|\zeta - \lambda|}{|\zeta - z|} = 2 \frac{|\lambda - z|}{1 - |z|^2}.
  \end{equation*}
\end{lemma}
\begin{proof}   By replacing $z$ with $\overline{\lambda} z$, we may assume that $\lambda = 1$.
  We map the unit disc onto the upper half plane via the map $w=i\frac{1+u}{1-u}$. Then $u= \frac{w-i}{w+i}$. Set $\alpha= i \frac{1+z}{1-z}$, then one computes
$$ \Big|\frac{u-1}{u-z} \Big|=  \Big|\frac{\alpha+i}{\alpha-{w}} \Big|.$$
The supremum over $w\in \R$ is attained for $w= \mathrm{Re}\ \alpha$ and it equals $\frac{|\alpha+i|}{|\mathrm{Im} \ \alpha|}$. Now we express this quantity in terms of $z$ and obtain the lemma.
\end{proof}

\begin{theorem}
  \label{thm:atomic-zero}
  Let $Z$ be a sequence in $\mathbb{D}$. Suppose that $Z$ can be partitioned as $Z = \bigcup_{k} Z_k$
  and that there exist $\lambda_k \in \mathbb{T}$ such that
  setting
  \begin{equation*}
    a_k = 2 \sum_{z \in Z_k} \frac{|\lambda_k - z|^2}{1 - |z|^2} \ \text{ and }  \mu = \sum_{k} a_k \delta_{\lambda_k},
  \end{equation*}
  we have that $\sum_k a_k<\infty$ and $S_\mu$ is an inner factor of $\mathcal D$.

  Then $Z$ is a zero set for the Dirichlet space.
\end{theorem}

\begin{proof}
  Since  $\{a_k\}$ is summable, the inequality $1 - |z| \le 2 \frac{|\lambda_k - z|^2}{1 - |z|^2}$ shows in particular that
  $Z$ is a Blaschke sequence. Let $B$ be the Blaschke product with zeros $Z$. Since  $S_\mu$ is a factor of $\mathcal{D}$, Proposition \ref{prop:inner_factor_char} yields a non-zero function $f \in \mathcal D$  with $S_\mu f \in \mathcal{D}$.
  We will show that $B f \in \mathcal{D}$.

  By Carleson's formula  it suffices to show that
  \begin{equation*}
    \sum_{z \in Z} \frac{1 - |z|^2}{|\zeta - z|^2} \le \sum_{k} \frac{2 a_k}{|\zeta - \lambda_k|^2}
    \quad \text{ for all } \zeta \in \mathbb{T}.
  \end{equation*}
  To see this, we fix $k$ and use Lemma \ref{lem:Poisson_sup} to estimate
  \begin{align*}
    \sum_{z \in Z_k} \frac{1 - |z|^2}{|\zeta - z|^2}
    &= \sum_{z \in Z_k} (1 - |z|^2) \frac{|\zeta - \lambda_k|^2}{|\zeta - z|^2} \frac{1}{|\zeta - \lambda_k|^2} \\
    &\le \sum_{z \in Z_k} \frac{|\lambda_k - z|^2}{1 - |z|^2} \frac{4}{|\zeta - \lambda_k|^2}
    = \frac{2 a_k}{|\zeta - \lambda_k|^2}.
  \end{align*}
  Summing in $k$ yields the desired inequality.
\end{proof}
For clarity we separate  a case of the previous theorem where each of the sets $Z_k=\{r_k\lambda_k\}$ consists of a single point.
The statement below follows from the previous theorem and Carleson's formula.
\begin{corollary} \label{cor:boundary_implies_zero} For $k\in \N$ let $r_k\in (0,1)$ and $\lambda_k \in \T$ such that $\{r_k \lambda_k\}\subseteq \D$ is a Blaschke sequence that is not a zero set for $\mathcal D$, and let $\mu= \sum_k (1-r_k)\delta_{\lambda_k}$. Then $S_\mu$ is not a factor of the Dirichlet space.
\end{corollary}

\begin{question}\label{ques:SingVsZeros}
  Is the converse of Corollary \ref{cor:boundary_implies_zero} true?\footnote{See Appendix \ref{sec:appendixb} for the answer.}
\end{question}

Note that there is no pointwise bound of the form $D_z(S_\mu) \lesssim D_z(B)$ for all $z \in \mathbb{T}$,
and it is generally not true that if $g$ is outer such that $B g \in \mathcal{D}$, then $S_\mu g \in \mathcal{D}$,
as the example of a single Blaschke factor shows. In fact, the failure persists even if we take $B g$
to be the extremal function of $B H^2\cap \mathcal{D}$, since the extremal function in the case of a single
Blaschke factor with zero $r \lambda$ does not vanish at $\lambda \in \mathbb{T}$; see Formula
(2.2) in \cite{GRS02}.

Nonetheless, many of our results about inner factors can be considered to be evidence in favor of a positive answer to Question \ref{ques:SingVsZeros}. Moreover, if $\mathcal{D}$ is replaced by $A^\infty$, then the analogous question has a positive answer, since the statements about zero sets and inner factors are both equivalent to $\overline{ \{\lambda_k  : k \in \mathbb{N} \} }$ being a Carleson set, see Theorems \ref{thm:CTWN} and \ref{thm:TW_singular} in the introduction.

\section{Pushing the points out}

In this section, we prove a result that will in particular imply Theorem \ref{thm:pushing_intro}. Throughout
this section, let $\mathcal{H}$ be a Hilbert space of analytic functions on $\mathbb{D}$ (i.e.\ with continuous point evaluations) such that
\begin{enumerate}
  \item $M_z: \mathcal{H} \to \mathcal{H}, f \mapsto z f$, is bounded,
  \item if $f \in \mathcal{H}, \lambda \in \mathbb{D}$ and $f(\lambda) = 0$, then $\frac{f}{z-\lambda} \in \mathcal{H}$, and
  \item $\|z f\| \ge \|f\|$ for all $f \in \mathcal{H}$.
\end{enumerate}
It is well known that the Dirichlet space $\mathcal{D}$ satisfies these properties.
Indeed, (1) and (3) are immediate from \eqref{eqn:Dirichlet_norm}.
The division property (2) follows, for instance, by applying \eqref{equ:D_z(I)} to see that if $f \in \mathcal{D}$
with $f(\lambda) = 0$, then $g = \frac{1 - \overline{\lambda} z}{z - \lambda} f \in \mathcal{D}$, and hence
$\frac{f}{z - \lambda} = \frac{1}{1 - \overline{\lambda} z} g \in \mathcal{D}$ because $\frac{1}{1 - \overline{\lambda} z}$
is holomorphic in a neighborhood of $\overline{\mathbb{D}}$ and hence a multiplier.

Recall from the previous section that when we say that a holomorphic function vanishes on a sequence, this means
vanishing with multiplicity on repeated points.

\begin{lemma}
  \label{lem:pushing_out_finite}
  Let $z_1,\ldots,z_N$ be a finite sequence of points in $\mathbb{D} \setminus \{0\}$ and let $t \in \mathbb{R}$.
  For $0 < r < 1$, let $\mathcal{M}_r = \{f \in \mathcal{H}: f \text{ vanishes on } z_1,\ldots,z_N \text{ and } r e^{it} \}$.
  Then $r \mapsto \gamma(\mathcal{M}_r)$ is non-decreasing.
\end{lemma}

\begin{proof}
  For ease of notation, we write $\gamma(r) = \gamma(\mathcal{M}_r)$.
  Let $Z = \{z_j\}_{j=1}^N$ and let $1 < s < \infty$. Hypotheses (1) and (2) on $\mathcal{H}$ imply that
  \begin{equation*}
    \gamma^{-2}\Big(\frac{1}{s}\Big) = \inf \{ \|(1 - s e^{-it} z) f \|^2: f \in \mathcal{H}, f(0) = 1, f \text{ vanishes on }Z \}.
  \end{equation*}
  If $f \in \mathcal{H}$ with $f(0) = 1$, let $h_f(s) = \|(1 - s e^{- it}z) f\|^2$.
  We will show that $h_f$ is non-decreasing in $s$, which implies the result.

  Indeed, for $s > 1$, we compute with the help of (3) and the Cauchy--Schwarz inequality
  \begin{equation*}
    h_f'(s) = - 2 \Re e^{it} \langle f,zf \rangle + 2 s \|z f\|^2 \ge - 2 \Re e^{it} \langle f, zf \rangle + 2 \|f\| \|z f\| \ge 0,
  \end{equation*}
  as desired.
\end{proof}

We require the following simple lemma, which in the case of the Dirichlet space is a consequence of Carleson's formula.
\begin{lemma}
  \label{lem:divide}
  Let $f \in \mathcal{H}, \lambda \in \mathbb{D}$ and $f(\lambda) = 0$. Let $g = \frac{1 - \overline{\lambda} z}{z - \lambda} f$.
  Then $g \in \mathcal{H}$ with $\|g\| \le \|f\|$.
\end{lemma}

\begin{proof}
  Assumptions (1) and (2) imply that $g \in \mathcal{H}$. Let $h = \frac{f}{z- \lambda}$, which belongs to $\mathcal{H}$.
  Then a small computation shows that
  \begin{equation*}
    \|(z - \lambda) h \|^2 - \|(1 - \overline{\lambda} z)h \|^2 = (\|z h \|^2 - \|h\|^2)(1 - |\lambda|^2) \ge 0
  \end{equation*}
  by Property (3). The lemma follows.
\end{proof}

The following result implies Theorem \ref{thm:pushing_intro}.
\begin{theorem}
  \label{thm:pushing_out}
  Let $\mathcal{H}$ be a Hilbert space of analytic functions on $\mathbb{D}$ satisfying (1)--(3) above.
  Let $Z=\{r_n e^{it_n}\}$ be a sequence in $\mathbb{D} \setminus\{0\}$ that is a zero set for $\mathcal{H}$, and for each $n$ let $0<r_n \le s_n <1$.

  Then $W=\{s_n e^{it_n}\}$ is a zero set for $\mathcal{H}$ with $\gamma_{\mathcal{H}}(W) \ge \gamma_{\mathcal{H}}(Z)$.
\end{theorem}

\begin{proof}
  For $N \in \mathbb{N}$, let $Z_N = \{r_n e^{ i t_n} \}_{n=1}^N$ and $W_N = \{s_n e^{ i t_n } \}_{n=1}^N$. Lemma \ref{lem:pushing_out_finite} and a straightforward induction argument show that
  \begin{equation*}
    \gamma_{\mathcal{H}} ( W_N) \ge \gamma_{\mathcal{H}}(Z_N) \ge \gamma_{\mathcal{H}}(Z).
  \end{equation*}
  Taking the limit $N \to \infty$ and applying Lemma \ref{lem:gamma_continuous} gives $\gamma_\mathcal{H}(W) \ge \gamma_{\mathcal{H}}(Z) > 0$.

  If $f$ denotes the extremal function of the subspace of functions vanishing on $W$, then Lemma \ref{lem:divide} implies
  that $f$ vanishes exactly on $W$, because otherwise, we could divide out extraneous zeros, contradicting maximality of $f$.
  Thus, $W$ is a zero set for $\mathcal{H}$.
\end{proof}

Notice that if $\{\lambda_n\}\subseteq \T$ and if $\mu=\sum_{n=1}^\infty a_n \delta_{\lambda_n}$ and $\nu= \sum_{n=1}^\infty b_n \delta_{\lambda_n}$ for summable sequences that satisfy $0<b_n\le a_n$, then by Carleson's formula it is obvious that $S_\nu$ is a factor of $\mathcal D$, whenever $S_\mu$ is. Thus, the qualitative part of Theorem \ref{thm:pushing_out} in the case of the Drichlet space also follows from a positive answer to Question \ref{ques:SingVsZeros}.

In the appendix,
we will construct an example of a radial complete Pick space on $\mathbb{D}$
in which Theorem \ref{thm:pushing_out} fails.

\section{Singular inner factors: rapidly decaying weights}

In this section, we will prove Theorem \ref{thm:SS_intro}, our analogue of the Shapiro--Shields theorem for singular inner factors. The key to proving this result is to understand the extremal function of the invariant subspace generated by an atomic singular inner function.

We start with an approximation result, which is a variant of Lemma \ref{lem:gamma_continuous} in the Dirichlet space.
Crucially, we do not assume that the sequence of subspaces is decreasing.

\begin{lemma}
  \label{lem:extremal_inner_convergence}
 Let $B_n,S$ be inner functions such that $B_n\to S$ locally uniformly in $\mathbb{D}$, $S(0)\ne 0$, and $SH^2\cap \mathcal D\ne (0)$.
 Assume that
\begin{itemize}
  \item there exists $c> 0$ such that $D_z(B_n)\le c \ D_z(S)$ for all $n\in \N$ and for all  $z\in \partial \mathbb{D}$, and
 \item   for all $z\in \partial \mathbb{D}$  we have $D_z(B_n)\to D_z(S)$ as $n\to \infty$.
  \end{itemize}
  Let $\varphi_n$ denote the extremal function for $B_n H^2\cap \mathcal D$ and $\varphi$ the extremal function for $SH^2\cap \mathcal D$.
  Then $\varphi_n\to \varphi$ in norm.
  In particular,
  \[
    \lim_{n \to \infty} \gamma(B_n H^2\cap \mathcal D) = \gamma(SH^2\cap \mathcal D).
  \]
\end{lemma}

\begin{proof}
  For ease of notation, we will write $\gamma(S) = \gamma(SH^2\cap \mathcal D)$ and $\gamma(B_n) = \gamma(B_n H^2\cap \mathcal D)$.
  Let $\varphi$ be the extremal function for $SH^2\cap \mathcal D$, and write $\varphi = S \psi$, where $\psi \in H^2$.
  Carleson's formula for the Dirichlet integral shows that whenever $If\in \mathcal D$ for some inner function $I$ and $f\in H^2$, then $f\in \mathcal D$ and \begin{align*}\|If\|^2=\int_{|z|=1}D_z(I)|f(z)|^2 \frac{|dz|}{2\pi}+ \|f\|^2; \end{align*}
  see Section \ref{sec:prelim}.
  Thus, the hypotheses and the dominated convergence theorem imply that $\|B_n\psi\|\to \|S\psi\| = 1$, so
  \begin{equation}
    \label{eqn:inner_convergence_1}
    \liminf_{n \to \infty} \gamma(B_n) \ge \liminf_{n \to \infty} \frac{\Re ( B_n(0) \psi(0))}{\|B_n \psi\|} = S(0) \psi(0) = \gamma(S).
  \end{equation}

  Conversely, let $\varphi_n = B_n \psi_n$ be the extremal function for $B_n H^2\cap \mathcal D$, where $\psi_n \in H^2$. By Carleson's formula, $\|\psi_n\| \le 1$ for all $n$,
  so there exists a subsequence $\{\psi_{n_k}\}$ that converges weakly to some $\psi \in \mathcal{D}$. Then $\varphi_{n_k} = B_{n_k} \psi_{n_k}$ converges weakly to $S \psi$; whence
  $\|S \psi\| \le 1$. It follows that
  \begin{equation*}
    \gamma(S) \ge S(0) \psi(0) = \lim_{k \to \infty} \varphi_{n_k}(0) = \lim_{k \to \infty} \gamma(B_{n_k}) \ge \gamma(S),
  \end{equation*}
  where the last inequality follows from \eqref{eqn:inner_convergence_1}. Thus, $\gamma(S) = S(0) \psi(0)$,
  so by uniqueness of the extremal function, $S \psi$ is the extremal function for $SH^2\cap \mathcal D$.

  This argument shows that every subsequence of $\{\varphi_n\}$ has a further subsequence that converges weakly to the extremal function $\varphi$ for $SH^2\cap \mathcal D$; hence $\{\varphi_n\}$ converges weakly to $\varphi$. Since $\|\varphi_n\| = 1$ for all $n$ and $\|\varphi\| = 1$,  the convergence is in norm.
\end{proof}

It follows from Carleson's formula for the Dirichlet integral (see Section \ref{sec:prelim})
that each atomic singular inner function $S(z) = \exp( - a \frac{1+z}{1-z})$, where $a > 0$, is an inner factor of $\mathcal{D}$. Indeed, $(1-z) S(z) \in \mathcal{D}$. We now compute the extremal function for the corresponding invariant subspace.

\begin{theorem}\label{ExtremalFunction} For $a>0$ let $S(z)= e^{-a\frac{1+z}{1-z}}$ and let $\varphi$ be the extremal function for $SH^2\cap \mathcal D$.
Then $$(z\varphi)'(z)= A^{-1/2}\ \frac{A-z}{1-z}\ e^{-a \frac{1+z}{1-z}},$$ where
$A= \frac{\int_0^1 th(t)dt}{\int_0^1h(t)dt}$ for $h(t)=\frac{e^{-a\frac{1+t}{1-t}}}{1-t}$.
\end{theorem}

Note that $0<A<1$ and if $a\to 0$, then $A\to 1$, because $\frac{1}{1-t}\notin L^1[0,1]$.

\begin{proof}
We will fix $a>0$ and use the notation as in the theorem. Choose $n\in \N$ so that  $r=1-\frac{a}{n}>0$, and set $B_n(z)= \left(\frac{r-z}{1-rz}\right)^n$.
We define $$\mathcal{M}_n=\{f\in \mathcal D: f(r)=f'(r)=\dots =f^{(n-1)}(r)=0\}$$
and we let $\varphi_n$ denote the extremal function for $\mathcal{M}_n$. Then $\mathcal{M}_n=B_nH^2\cap \mathcal D$. In particular, we write $\varphi_n=B_n\psi_n$ for some $\psi_n\in H^2$.

Using the definition of $r=r(n)$ we obtain \begin{align}\label{equ:F(z,n)}B_n(z)= \left(1-\frac{a}{n}\frac{1+z}{1-z}+O\Big(\frac{1}{n^2}\Big)\right)^n. \end{align} Here  we have used $O(\frac{1}{n^2})$ for a remainder term that has modulus $\le C/n^2$ for all $n$, where $C>0$ may depend on $z$.
Hence $B_n \to S$ locally uniformly on $\mathbb{D}$ as $n\to \infty$.
Moreover, by \eqref{eqn:local_Diri_inner},
\begin{align*}D_z(B_n)&=  n \frac{1-r^2}{|1-rz|^2}= \frac{2a(1-\frac{a}{2n})}{|1-rz|^2}\le  \frac{8a}{|1-z|^2}=4D_z(S) \text{ and }\\ D_z(B_n)&\to D_z(S) \text{ as }n\to \infty.
\end{align*} Hence Lemma \ref{lem:extremal_inner_convergence} implies that $\varphi_n \to \varphi$ in norm. (This approximation idea has also been  useful in the context of Bergman spaces; see for instance the discussion at the end of Chapter 5 in \cite{DS04}.)

Next, we compute $\varphi_n$.
For $j=0, 1,\dots$ let $k^{(j)}_w(z)$ be the reproducing kernel for $\mathcal D$ for the $j$th derivative at $w$. Let $P_n$ be the orthogonal projection onto $\mathcal{M}_n$. Then $1-P_n 1\in \mathcal{M}_n^\perp$, hence $f_n:=P_n 1=1-\sum_{j=0}^{n-1}b_j k^{(j)}_r$ for some $b_0, \dots, b_{n-1}\in \C$. The properties of the projections imply that $\|f_n\|^2=\la f_n,1\ra=f_n(0)$, hence $f_n/\sqrt{f_n(0)}$ is a unit vector in $\mathcal{M}_n\ominus z\mathcal{M}_n$, which is $>0$ at 0. It follows that $\varphi_n=f_n/\sqrt{f_n(0)}$.

Note that $(zf_n)'=1-\sum_{j=0}^{n-1}b_j (zk_r^{(j)})'$ and one verifies by induction that for $j\ge 0$ we have $(zk_r^{(j)})'= \frac{j!z^j}{(1-rz)^{j+1}}$.
 Hence
$$(zf_n)'=1-\sum_{j=0}^{n-1} \frac{b_{j} j! z^{j}}{(1-rz)^{j+1}}= \frac{p(z)}{(1-rz)^n},$$ where $p$ is a polynomial of degree $n$ such that $\lim_{|z|\to \infty} \frac{p(z)}{(1-rz)^n}=1$. Since $f_n\in \mathcal{M}_n$ we note that $(zf_n)'$ and hence $p$, has a zero of multiplicity $n-1$ at $r$. Thus,
\begin{align}\label{Equ:zf_n prime}(zf_n)'(z)= r^{n} \frac{A_n-z}{1-rz}\left(\frac{r-z}{1-rz}\right)^{n-1}\end{align} for some $A_n\in \C$. In fact,  $A_n$ is determined by $f_n(r)=0$. Since $$f_n(r)=\int_0^1 (zf_n)'(tr)dt= r^{n}\int_0^1 \frac{A_n-rt}{1-r^2t}\left(\frac{r-rt}{1-r^2t}\right)^{n-1}dt$$ we have
$$A_n=r \frac{\int_0^1th_n(t) dt}{\int_0^1 h_n(t)dt},$$ where  $h_n(t)= \frac{r^{n-1}}{1-r^2t}\left(\frac{1-t}{1-r^2t}\right)^{n-1}$.

Then $ h_n(1)=0$ for each $n$, and one calculates that $$h_n'(t)= -r^{n-1}\frac{r^2t+n(1-r^2)-1}{(1-r^2t)^2 (1-t)} \left(\frac{1-t}{1-r^2t}\right)^{n-1}.$$
Recall that $r=1-\frac{a}{n}$, hence for $t>1-a$ we have
$$r^2t+n(1-r^2)-1= r^2 t +a(1+r)-1>r^2t+a(1+r^2) -1>r^2 +a-1> a(1-\frac{2}{n}).$$ This shows that $h_n'(t)<0$ for sufficienly large $n$ and $t\in [0,1)$ with $t>1-a$.

Similarly, we obtain $$\frac{1-t}{1-r^2t}= 1-t\frac{1-r^2}{1-r^2t} = 1-\frac{2at}{n}\frac{1-\frac{a}{2n}}{1-t+\frac{2at}{n}(1-\frac{a}{2n})} = 1 - \frac{2 a t}{n(1-t)} + O \Big( \frac{1}{n^2} \Big).$$

Hence for each $t\in [0,1)$ we have $$h_n(t)\to \frac{e^{-a}}{1-t}e^{-\frac{2at}{1-t}}= \frac{e^{-a\frac{1+t}{1-t}}}{1-t}=:h(t).$$ In fact, since for sufficiently large $n$ each $h_n$ decreases to $0$ in a fixed interval near 1, one checks  that the sequence $\{h_n\}$ is uniformly bounded on $[0,1)$.

This implies that
$$\lim_{n\to \infty}A_n=A= \frac{\int_0^1 th(t)dt}{\int_0^1h(t)dt}. $$

\

Next we use  equations (\ref{equ:F(z,n)}) and (\ref{Equ:zf_n prime}) to conclude that $$(zf_n)'(z) \to e^{-a}\frac{A-z}{1-z}e^{-a \frac{1+z}{1-z}}$$ for all $z\in \mathbb{D}$. Note that $f_n(0)=(zf_n)'(0)$. Hence  $(z\varphi_n)'(z) =\frac{(zf_n)'(z)}{\sqrt{f_n(0)}}\to A^{-1/2} \frac{A-z}{1-z}e^{-a \frac{1+z}{1-z}}$. Since we already established that $\varphi_n\to \varphi$ locally uniformly, this proves the theorem.
\end{proof}

Based on Theorem \ref{ExtremalFunction}, one can express the extremal function $\varphi$ itself with the help of the exponential integral $E_1$, but we will not do so here.
What is important for us is the value of $\varphi$ at $0$.

\begin{corollary}
  \label{cor:atomic_gamma}
  Let $\lambda \in \mathbb{T}$, let $a > 0$, let $S_a(z) = \exp({ -a \frac{\lambda+z}{\lambda-z}})$, and let $\varphi_a$ be the extremal function for $S_aH^2\cap \mathcal D$.
  Then
  \[
    1-\varphi_a(0)=1 - \gamma(S_a H^2 \cap \mathcal{D}) \sim \frac{1}{ 2 \log\frac{1}{a}}  \quad \text{ as } a \to 0.
  \]
\end{corollary}

\begin{proof}
  By rotation invariance, we may assume that $\lambda = 1$.
  By Theorem \ref{ExtremalFunction}, we have
  \begin{equation*}
    \gamma(S_aH^2 \cap \mathcal{D}) = A^{1/2} e^{-a},
  \end{equation*}
  where
  $A= \frac{\int_0^1 th(t)dt}{\int_0^1h(t)dt}$ for $h(t)=\frac{e^{-a\frac{1+t}{1-t}}}{1-t}$,
  and $A \to 1$ as $a \to 0$. Hence
  \begin{equation}
    \label{eqn:atomic_gamma}
    1 - \gamma(S_aH^2 \cap \mathcal{D}) \sim \frac{1}{2} (1 - A e^{-2 a}) = \frac{1}{2} (1 - A) + O(a).
  \end{equation}
  Since $h(t)=e^{-a}\frac{e^{\frac{-2a}{1-t}}}{1-t}$, substituting $s=\frac{a}{1-t}$ and integrating by parts, we obtain
$$1-A= \frac{\int_0^1 e^{\frac{-2a}{1-t}}dt}{\int_0^1 \frac{e^{\frac{-2a}{1-t}}}{1-t}dt}=\frac{a\int_a^\infty e^{-2s}\frac{1}{s^2}ds}{\int_a^\infty e^{-2s}\frac{1}{s}ds}=\frac{a\left(\frac{e^{-2a}}{a}-2\int_a^\infty e^{-2s}\frac{1}{s}ds\right)}{\int_a^\infty e^{-2s}\frac{1}{s}ds}.$$
Furthermore, we have
\begin{align*}\int_a^\infty \frac{e^{-2s}}{s}ds&=\int_a^1 \frac{1}{s}ds+ \int_a^1 \frac{e^{-2s}-1}{s}ds +\int_1^\infty \frac{e^{-2s}}{s}ds\\
&= \log \frac{1}{a} + \int_a^1 \frac{e^{-2s}-1}{s}ds +\int_1^\infty \frac{e^{-2s}}{s}ds \\
&= \log \frac{1}{a} + O(1).
\end{align*}
Hence \begin{align*}1-A \sim \frac{1}{\log\frac{1}{a}}. \end{align*}
Combining this identity with \eqref{eqn:atomic_gamma} yields the result.
\end{proof}

We are now ready to prove the implication (iii) $\Rightarrow$ (i) in Theorem \ref{thm:SS_intro}.

\begin{proposition}
  \label{prop:SS_singular_suff}
  Let $\{a_n\}$ be a sequence of positive numbers such that
  \begin{equation*}
    \sum_n \frac{1}{\log( 1 +  \frac{1}{a_n})} < \infty
  \end{equation*}
  and let $\{\lambda_n\}$ be a sequence in $\mathbb{T}$. Then the singular inner function $S_\mu$, where
  \begin{equation*}
    \mu = \sum_n a_n \delta_{\lambda_n},
  \end{equation*}
  is a factor of $\mathcal{D}$.
\end{proposition}

\begin{proof}
  Let $\mu_n = a_n \delta_{\lambda_n}$.
  The assumption implies that $a_n \to 0$, so we may assume that $a_n \in (0,\frac{1}{2})$ for all $n$. Then by Corollary \ref{cor:atomic_gamma},
  \begin{equation*}
    1 - \gamma(S_{\mu_n} H^2 \cap \mathcal{D}) \sim \frac{1}{2 \log \frac{1}{a_n}},
  \end{equation*}
  and this quantity is summable by our assumption.
  Thus, $S_\mu$ is an inner factor of $\mathcal{D}$ by Lemma \ref{lem:infinite_products}.
\end{proof}

The other implications of Theorem \ref{thm:SS_intro} and sharpness of Proposition \ref{prop:SS_singular_suff} follow
from sharpness of the Shapiro--Shields condition for zero sets, Corollary \ref{cor:boundary_implies_zero}
and Theorem \ref{thm:zero_subset}.
Since the implication (i)$\Rightarrow$(ii) is trivial the following result concludes the proof of Theorem \ref{thm:SS_intro}.

\begin{proposition}
  Let $\{a_n\}$ be a summable sequence of positive numbers such that
  \begin{equation*}
    \sum_n \frac{1}{\log(1 + \frac{1}{a_n})} = \infty.
  \end{equation*}
  Then there exists a sequence $\{\lambda_n\}\to 1$ in $\mathbb{T}$ such that defining
  $\mu = \sum_n a_n \delta_{\lambda_n}$, the singular inner function $S_\mu$ is not a factor in the Dirichlet space.
\end{proposition}

\begin{proof}
  We may assume that $a_n \in (0,1)$ for all $n$. Let $r_n = 1 - a_n$.
  Then by the Rudin--Nagel--Shapiro result \cite{NRS82} (see also \cite[Theorem 4.2.2]{EKM+14}), there exist $w_n \in \mathbb{T}$ such that $\{r_n w_n\}$ is not a zero set for $\mathcal{D}$.
  Here, some of the $w_n$'s are allowed to be the same.
  Theorem \ref{thm:zero_subset} shows that there is a subsequence $\{r_{n_j} w_{n_j}\}$ that is not a zero set for $\mathcal{D}$, and such that $\{r_{n_j} w_{n_j} \}$ converges to some $\lambda\in \T$. By the rotational invariance of $\mathcal D$ we may assume $\{r_{n_j} w_{n_j} \}\to 1$. Then Corollary \ref{cor:boundary_implies_zero} shows that $S_{\mu_1}$ is not a factor of $\mathcal D$, where $\mu_1=\sum_{j=1}^\infty a_{n_j} \delta_{w_{n_j}}$. Finally, set $\lambda_{n_j}=w_{n_j}$ and use any sequence in $\T$ that converges to 1 to define $\mu_2 = \sum_{k=1, k\notin \{n_j\}}^\infty a_{k} \delta_{\lambda_{k}}$. Then by Corollary \ref{cor:inner_factor_product} $S_{\mu_1+\mu_2}$ cannot be a factor of $\mathcal D$.
\end{proof}

\section{Sufficient conditions for zero sets}

The results of the previous section imply results about zero sets.  We start with a more precise version of Theorem \ref{thm:zero_decomp_intro}.

\begin{theorem}
  \label{thm:zero_decomp}
  Let $Z$ be a sequence in $\mathbb{D}$. Suppose that $Z$ can be partitioned as $Z = \bigcup_{k} Z_k$
  and that there exist $\lambda_k \in \mathbb{T}$ such that
  setting
  \begin{equation*}
    a_k = 2 \sum_{z \in Z_k} \frac{|\lambda_k - z|^2}{1 - |z|^2},
  \end{equation*}
  we have
  \begin{equation*}
    \sum_{k} \frac{1}{\log(1 + \frac{1}{a_k})} < \infty.
  \end{equation*}
  Then $Z$ is a zero set for the Dirichlet space.
\end{theorem}

\begin{proof}
  Since $a_k \le \frac{1}{\log(1 + \frac{1}{a_k})}$ the assumption implies that $\{a_k\}$ is summable. Thus, the assumption $ \sum_{k} \frac{1}{\log(1 + \frac{1}{a_k})} < \infty$ implies the theorem by Theorems \ref{thm:SS_intro} and \ref{thm:atomic-zero}.
\end{proof}

Theorem \ref{thm:zero_decomp_intro} is now a consequence of the last result.

\begin{proof}[Proof of Theorem \ref{thm:zero_decomp_intro}]
  Let $b_k = 2 K(k)^2 \sum_{z \in Z_k} (1 - |z|)$ denote the quantity in the statement of Theorem \ref{thm:zero_decomp_intro}. Then the quantity in Theorem \ref{thm:zero_decomp} satisfies
  \begin{equation*}
    a_k = 2 \sum_{z \in Z_k} \frac{|\lambda_k - z|^2}{1 - |z|^2}
    \le 2 K(k)^2 \sum_{z \in Z_k}  \frac{( 1 - |z|)^2}{1 - |z|^2}
    \le b_k.
  \end{equation*}
  Thus, Theorem \ref{thm:zero_decomp} shows that $Z$ is a zero set.
\end{proof}

Theorem \ref{thm:zero_decomp} can be improved slightly by using the fact that moving
zeros away from the origin preserves zero sets; see Theorem \ref{thm:pushing_out}.
The idea is that $\frac{|\lambda_k - z|^2}{1 - |z|^2}$ can sometimes be decreased by replacing $z$ with a point that lies on the same radius as $z$. For instance, the new point could lie in a nontangential approach region; see Figure \ref{fig:push_Stolz}.

\begin{figure}[ht]
\begin{tikzpicture}[scale=0.8, line cap=round, line join=round]
\def\R{3}          
\def\K{1.3}        
\def\alpha{16}     

\pgfmathdeclarefunction{rstolz}{1}{%
  \pgfmathparse{ ((\K*\K) - cos(#1) - sqrt(((\K*\K) - cos(#1))^2 - ((\K*\K) - 1)^2))
                 / ((\K*\K) - 1) }%
}

\draw[line width=0.8pt, blue] (0,0) circle (\R);

\fill[gray!40, even odd rule]
  plot[domain=-180:180, smooth, samples=360, variable=\t]
    ({\R * rstolz(\t) * cos(\t)}, {\R * rstolz(\t) * sin(\t)}) -- cycle;

\draw[line width=0.9pt, red!70!black]
  plot[domain=-180:180, smooth, samples=360, variable=\t]
    ({\R * rstolz(\t) * cos(\t)}, {\R * rstolz(\t) * sin(\t)}) -- cycle;

\coordinate (Bdry) at ({\R*cos(\alpha)}, {\R*sin(\alpha)});
\draw[thin] (0,0) -- (Bdry);

\coordinate (w) at ({\R*rstolz(\alpha)*cos(\alpha)}, {\R*rstolz(\alpha)*sin(\alpha)});
\coordinate (z) at ({\R*((rstolz(\alpha)+1)/2)*cos(\alpha)}, {\R*((rstolz(\alpha)+1)/2)*sin(\alpha)});

\fill (w) circle (0.06);
\fill (z) circle (0.06);
\node[below] at (w) {$w_n$};
\node[above] at (z) {$z_n$};

\end{tikzpicture}
  \caption{}
  \label{fig:push_Stolz}
\end{figure}


More precisely, in the setting of Theorem \ref{thm:zero_decomp}, let $r_z \in (0,1]$ for each $z \in Z$ and
consider
\begin{equation*}
  \widetilde{a}_k = 2 \sum_{z \in Z_k} \frac{|\lambda_k - r_z z|^2}{1 - r_z^2 |z|^2}.
\end{equation*}
If
\begin{equation*}
  \sum_k \frac{1}{\log(1 + \widetilde{a}_k^{-1})} < \infty,
\end{equation*}
then Theorem \ref{thm:zero_decomp} shows that $\{r_z z: z \in Z \}$ is a zero set,
and hence $Z$ is a zero set by Theorem \ref{thm:pushing_out}.
The optimal choice of $r_z$ amounts to solving an optimization problem.
The following lemma gives the solution in the case $\lambda_k = 1$;  the general case is obtained
by applying a rotation.

\begin{lemma} \label{lem:minOfg} For $t\in [-\frac{\pi}{2},\frac{\pi}{2}]$ define $g(r)=\frac{|1-re^{it}|^2}{1-r^2}= \frac{1-2r\cos(t)+r^2}{1-r^2}$. Then $g$ is decreasing in $[0,\frac{\cos t}{1+|\sin t|})$ and increasing in $(\frac{\cos t}{1+|\sin t|},1)$ and  $$\min\{g(r):r\in [0,1)\}=|\sin t|.$$
\end{lemma}
We omit the elementary proof.

In case $\lambda_k =1$, write $Z_k = \{z_1,z_2,\ldots\}$ and $z_n = |z_n| e^{i t_n}$,
and let
\begin{equation*}
 M= \Big\{n: t_n \in \Big(-\frac{\pi}{2},\frac{\pi}{2}\Big) \text{ and } |z_n| > \frac{\cos t_n}{1+|\sin t_n|} \Big\}
\end{equation*}
Then Lemma \ref{lem:minOfg} shows that
we can replace $a_k$ by the smaller quantity
\begin{equation*}
  \widetilde{a}_k = \sum_{n \notin M} 2 \frac{|1 - z_n|^2}{1 - |z_n|^2} + \sum_{n \in M} 2 | \sin t_n|.
\end{equation*}

Frostman shifts are another way of producing Blaschke products from singular inner functions.
We record the observation that the Frostman shift of a factor of $\mathcal{D}$ is again a factor of $\mathcal{D}$.
We are grateful to John M\textsuperscript{c}Carthy for raising this question.

\begin{proposition}
  Let $I$ be an inner function that is a factor of $\mathcal{D}$. Then for every biholomorphic automorphism $\theta$ of $\mathbb{D}$,
  the inner function $\theta \circ I$ is a factor of $\mathcal{D}$.
\end{proposition}

\begin{proof}
  We may assume that $\theta(z) = \frac{z - a}{1 - \overline{a}z}$ for some $a \in \mathbb{D}$.
  Let $g \in \mathcal{D} \setminus \{0\}$ such that $I g \in \mathcal{D}$. Then $(1 - \overline{a} I) g \in \mathcal{D}$ and
  \begin{equation*}
    (\theta \circ I) \cdot (1 - \overline{a} I) g = (I - a) g \in \mathcal{D},
  \end{equation*}
  so $\theta \circ I$ is a factor of $\mathcal{D}$.
\end{proof}

\section{Rate of convergence conditions}

Let $\{z_n\}$ be a sequence in $\mathbb{D}$ converging to $1$. It is natural to seek sufficient conditions for $\{z_n\}$ to be a zero set
for the Dirichlet space that only depend on the rate of convergence of $\{z_n\}$ to $1$; see e.g. \cite{MS09}.
The following result shows, roughly speaking, that any such sufficient condition
is necessarily implied by the Carleson set condition; see Theorem \ref{thm:CTWN} and the discussion following it.
The proof is inspired by that of \cite[Theorem 4.5.2]{EKM+14}.

\begin{theorem}
  \label{thm:rate_of_convergence}
  Let $\Omega \subset \mathbb{D}^{\mathbb{N}}$ be a set of sequences with the following two properties:
  \begin{enumerate}
    \item every sequence in $\Omega$ is a zero set for the Dirichlet space, and
    \item whenever $\{z_n\} \in \Omega$ and $\{w_n\} \in \mathbb{D}^{\mathbb{N}}$ satisfies
      $|1 - w_n| \le |1 - z_n|$ for all $n \in \mathbb{N}$, then $\{w_n\} \in \Omega$.
  \end{enumerate}
  Then for all $\{z_n\} \in \Omega$, the set
  \begin{equation}
    \label{eqn:proj_boundary}
    \overline{\Big\{ \frac{z_n}{|z_n|}: n \in \mathbb{N}, z_n \neq 0 \Big\}}
  \end{equation}
  is a Carleson set.
\end{theorem}
Theorem \ref{thm:CTWN} then implies that all sequences in $\Omega$ must be zero sets for $A^\infty$.
\begin{proof}
  We first show that every sequence $\{z_n\}$ in $\Omega$ converges to $1$. Suppose otherwise.
  Then there exist a subsequence $(z_{n_k})$ and $\varepsilon \in (0,1)$ with $|1 - z_{n_k}| \ge \varepsilon$
  for all $k$. We may find a sequence $\{w_n\}$ such that
  \begin{itemize}
    \item $w_{n_k} \in [0,1)$ for all $k$,
    \item $|1 - w_{n_k}| \ge \varepsilon$ for all $k$, and
    \item $|1 - w_{n}| \le |1 - z_n|$ for all $n \in \mathbb{N}$.
  \end{itemize}
  (For instance, choose $w_{n_k} = 1 - \min(1,|1 - z_{n_k}|)$ and $w_n = z_n$ for all remaining indices $n$.)
  By the second property of $\Omega$, we have $\{w_n\} \in \Omega$. But $|w_{n_k}| \le 1 - \varepsilon$ for all
  $k \in\mathbb{N}$, so $\{w_n\}$ cannot be a zero set for $\mathcal{D}$. This contradicts the first property of $\Omega$.
  Thus, every sequence in $\Omega$ converges to $1$.

  We now show the main statement. Assume towards a contradiction that $\{z_n\} \in \Omega$,
  but that the set in \eqref{eqn:proj_boundary} is not a Carleson set.
  We will obtain our contradiction by constructing a uniqueness sequence $\{w_n\}$ for the Dirichlet space
  with $|1 - w_n| \le |1 -z_n|$ for all $n \in \mathbb{N}$.

  Write $z_n = r_n e^{i \theta_n}$, where $r_n \to 1$ and $\theta_n \to 0$ by the first paragraph.
  Since the union of two Carleson sets is a Carleson set, we may assume by passing to a subsequence
  that either $\theta_n \ge 0$ or $\theta_n \le 0$ for all $n \in \mathbb{N}$. By replacing $z_n$ with $\overline{z_n}$ if necessary, we may further assume that $\theta_n \ge 0$ for all $n \in \mathbb{N}$. Furthermore, by possibly reordering the terms of the sequence we may suppose that $0<\theta_{n+1} \le \theta_n$ for all $n$ (we reorder the $w_n$ correspondingly).
  Finally, by passing to another subsequence, we may achieve that $0 < \theta_{n+1} < \theta_n < 1$ for all $n \in \mathbb{N}$. (Passing to a subsequence is allowed since we simply set $w_n = z_n$ for all points not in the subsequence, and deleting repeated $\theta_n$'s does not affect the set defined in \eqref{eqn:proj_boundary}.)

  Now, let $\tau_n = \frac{\theta_n}{2 \pi}$, $\delta_n = \tau_n - \tau_{n+1}$, $s_n = 1 - \delta_n$
  and $w_n = s_n e^{i \tau_n}$.
  Then
  \begin{align*}
    |1 -w_n| \le |1 - e^{i \tau_n}| + (1 - s_n) \le 2 \tau_n = \frac{\theta_n}{\pi} &\le \frac{1}{2} | 1 - e^{i \theta_n}| \\
                                                                                    &\le | 1 - r_n e^{ i \theta_n}| \\
                                                                                    &= |1 - z_n|.
  \end{align*}
To see that $\{w_n\}$ is a uniqueness set, note that since $\overline{ \{e^{ i \theta_n}: n \in \mathbb{N} \}}$
is not a Carleson set, we have with $\varepsilon_n = \theta_n - \theta_{n+1}$ that
\begin{equation*}
  \sum_n \varepsilon_n \log \Big( \frac{1}{\varepsilon_n} \Big) = \infty.
\end{equation*}
Here we have used the well-known fact that closed sets of measure 0 are Carleson sets, if and only if they have finite entropy, i.e. $\sum_n |I_n|\log \frac{1}{|I_n|}<\infty$, where $I_n$ are the complementary intervals of the set and $|I_n|$ denotes their normalized length, see \cite{Beurling} or \cite[Exercise 4.4.2]{EKM+14}.
Since $\delta_n = \frac{\varepsilon_n}{2 \pi}$, we conclude that
\begin{equation*}
  \sum_n \delta_n \log \Big( \frac{1}{\delta_n} \Big) = \infty.
\end{equation*}
In this setting, \cite[Theorem 4.1.6]{EKM+14} shows that $\{w_n\}$ is a uniqueness set, as desired.
\end{proof}

\section{Singular inner factors: Carleson sets}

We begin this section by proving a slightly more general version of Proposition \ref{prop:Carleson_singular_intro}.

\begin{proposition}
  The following are equivalent for a Borel set $E \subset \mathbb{T}$.
  \begin{enumerate}[label=\normalfont{(\roman*)}]
    \item Every singular inner function associated with a measure concentrated on $E$ is a factor of $\mathcal{D}$;
    \item $\overline{E}$ is a Carleson set.
  \end{enumerate}
\end{proposition}

\begin{proof}
  (ii) $\Rightarrow$ (i) This follows from the theorem of Taylor and Williams mentioned in the introduction (Theorem \ref{thm:TW_singular}).

  (i) $\Rightarrow$ (ii)
  Suppose that $\overline{E}$ is not a Carleson set.
  Using \cite[Lemma 4.5.3]{EKM+14} and arguing as in the proof of \cite[Theorem 4.5.2]{EKM+14},
  it suffices to consider the case when $E$ is given by a sequence $(e^{i \theta_n})_n$,
  where $(\theta_n)$ is a strictly decreasing sequence in $(0,1)$ tending to $0$.
  Let $\varepsilon_n = \theta_n - \theta_{n+1}$. Then $\sum_{n=1}^\infty \varepsilon_n < \infty$.
  Since $\overline{E}$ is not Carleson, $\sum_{n=1}^\infty \varepsilon_n \log(\frac{1}{\varepsilon_n}) =\infty$.

  Let
  \begin{equation*}
    \mu = \sum_{n=1}^\infty \varepsilon_n \delta_{e^{i \theta_n}}.
  \end{equation*}
  We claim that $S_\mu$ is not a factor of $\mathcal{D}$.
  Indeed, according to \cite[Theorem 4.1.6]{EKM+14},
  the conditions $\sum_{n=1}^\infty \varepsilon_n < \infty$ and $\sum_{n=1}^\infty \varepsilon_n \log( \frac{1}{\varepsilon_n}) = \infty$ imply that $\{(1 - \varepsilon_n) e^{i \theta_n}: n \in \mathbb{N} \}$ is a uniqueness set. Hence $S_\mu$ is not a factor of $\mathcal{D}$ by Corollary \ref{cor:boundary_implies_zero}.
\end{proof}

If $S_\mu$ is a singular inner factor of $\mathcal{D}$, then $\mu$ need not be concentrated
on a set whose closure is a Carleson set.
For instance, Theorem \ref{thm:SS_intro} shows that for every countable set $E \subset \mathbb{T}$,
there exists a singular measure $\mu$ such that each point of $E$ is an atom for $\mu$
and such that $S_\mu$ is a factor of $\mathcal{D}$. In particular,
it is possible that the support of $\mu$ is all of $\mathbb{T}$.

Nonetheless, we can still obtain some necessary conditions for $S_\mu$ to be a factor of $\mathcal{D}$.
It is known that if $\{z_n\}$ is a zero set for $\mathcal{D}$ and
\begin{equation*}
  V(z) =  \sum_n \frac{1 - |z_n|^2}{|z - z_n|^2},
\end{equation*}
then $\log^+ V \in L^1(\mathbb{T})$; see e.g. \cite[Theorem 4.1.5]{EKM+14}.
The following is the analogous condition for singular inner factors.
We argue as in the proof of \cite[Theorem 4.1.5]{EKM+14}.

\begin{lemma}
  \label{lem:singular_necessary}
  Let $\mu \in M^+(\mathbb{T})$ be a singular measure with associated singular inner function $S_\mu$.
  Let
  \begin{equation*}
    V_\mu(z) = \int_{\mathbb{T}} \frac{1}{|z- w|^2} \, d \mu(w).
  \end{equation*}
  If $S_\mu$ is a factor of $\mathcal{D}$, then
   $\log^+ V_\mu \in L^1(\mathbb{T})$.
\end{lemma}

\begin{proof}
  Let $f \in H^2 \setminus \{0\}$ be such that $S_\mu f \in \mathcal{D}$.
  By Carleson's formula (see Section \ref{sec:prelim}),
  \begin{equation*}
    \int_{\mathbb{T}} \int_{\mathbb{T}} \frac{2}{|z - w|^2}  |f(z)|^2 \, d \mu(w) \, |d z| < \infty,
  \end{equation*}
  so using that $\log^+ t \le t$, we find that
  \begin{equation*}
    \int_{\mathbb{T}} \log^+( V_\mu(z) |f(z)|^2) |d z| < \infty.
  \end{equation*}
  On the other hand, $\int_{\mathbb{T}} \log^+ |f(z)|^{-2} |d z| < \infty$,
  so since $\log^+(a b) \le \log^+ a + \log^+ b$, the result follows.
\end{proof}

We now obtain the following necessary condition from results of Ivrii and Ahern--Clark.

\begin{theorem}
  Let $\mu$ be a singular measure such that $S_\mu$ is a factor of $\mathcal{D}$.
  Then $\mu$ is concentrated on a countable union of Carleson sets.
\end{theorem}

\begin{proof}
  By Lemma \ref{lem:singular_necessary}, we have $\log^+ V_\mu \in L^1$. In this setting, \cite[Corollary 4]{AC74} shows
  that $S_\mu'$ belongs to the Nevanlinna class. Then the conclusion follows from \cite[Corollary 1.9]{Ivrii19}.
\end{proof}

\begin{remark}
  The condition in Lemma \ref{lem:singular_necessary} is not sufficient for $S_\mu$ to be a factor of $\mathcal{D}$. To see this, we first consider the analogous condition for zero sets mentioned before Lemma \ref{lem:singular_necessary}.
A result of Protas \cite{Protas73}, see also \cite[Theorem 7]{AC74}, shows that if $\sum_{n} (1 - |z_n|)^t < \infty$ for some $t \in (0,\frac{1}{2})$, then the Blaschke product $B$
with zeros $\{z_n\}$ satisfies $B' \in H^{1-t}$.
In particular, $B' \in N^+$, from which it follows that $\log^+ V \in L^1(\mathbb{T})$; see \cite[Corollary 4]{AC74}.
On the other hand, if $\sum_{n} \frac{1}{ \log ( \frac{1}{1 - |z_n|}) } = \infty$, then by the result of Nagel--Rudin--Shapiro mentioned in the introduction, we may choose the arguments of the $z_n$ in such a way that $\{z_n\}$ is a set
of uniqueness for $\mathcal{D}$.

Now, let $\{z_n\}$ be a sequence as in the preceding paragraph,
that is, a set of uniqueness with $\sum_n (1 - |z_n|)^t < \infty$ for some $t \in (0,\frac{1}{2})$.
Write $z_n = r_n e^{i t_n}$
and consider the singular measure
$\mu = \sum_{n} (1 - r_n) \delta_{e^{i t_n}}$. Since $\{z_n\}$ is not a zero set,
Corollary \ref{cor:boundary_implies_zero}  implies that $S_\mu$ is not a factor.
On the other hand, since $\sum_{n} (1 - r_n)^t < \infty$, a result of Caughran and Shields \cite{CS69}
shows that $S_\mu' \in H^t$.
Indeed, a proof of this result can be based on \cite[Theorem 5]{AC74}; see also the discussion
at the end of \cite{Ahern79}.
So again by \cite[Corollary 4]{AC74},
we have $\log^+ V_\mu \in L^1(\mathbb{T})$.
\end{remark}

\section{New singular inner factors}

We know the following sufficient conditions for singular inner factors in the Dirichlet space
from Theorem \ref{thm:SS_intro} and Proposition \ref{prop:Carleson_singular_intro} :
\begin{itemize}
  \item If $E \subset \mathbb{T}$ is a Carleson set, then for every measure $\mu$ supported on $E$, the singular inner function
    $S_\mu$ is a factor in the Dirichlet space.
  \item If $(\lambda_n)$ is any sequence in $\mathbb{T}$ and $(a_n)$ is a sequence in $(0,\infty)$
    satisfying $\sum_n \frac{1}{\log ( 1+ \frac{1}{a_n})} < \infty$ and $\mu = \sum_{n} a_n \delta_{\lambda_n}$,
    then $S_\mu$ is a factor in the Dirichlet space.
\end{itemize}

Our goal is to construct singular inner factors that cannot be achieved by combining these two sufficient conditions.

In the sequel, we write

\begin{equation*}
  \gamma(\mu) = \gamma(S_\mu H^2 \cap \mathcal{D}).
\end{equation*}

\begin{lemma}
  \label{lem:Carleson_gamma_zero}
  Let $E \subset \mathbb{T}$ be a countable Carleson set. Let $\mu$ be a positive measure supported on $E$.
  Then $\lim_{a \searrow 0} \gamma(a \mu) = 1$.
\end{lemma}

\begin{proof}
  Let
  \begin{equation*}
    \mathcal{M} = \overline{\bigcup_{a >0} S_{a \mu} H^2 \cap \mathcal{D}}.
  \end{equation*}
  We first show that $\mathcal{M} = \mathcal{D}$.
  Since $\mathcal{M}$ is the closure of an increasing union of invariant subspaces, it is an invariant
  subspace, so it suffices to show that $\mathcal{M}$ contains a cyclic function.

  Since $E$ is a Carleson set, there exists an outer function $f \in A^\infty$ such that $f S_\mu \in \mathcal{D}$ and such that
  the boundary zero set of $f$ is contained in $E$. (This follows from a more precise version of Theorem \ref{thm:TW_singular}; see Theorem 3.3 and the proof of Theorem 4.7 in \cite{TW70}.) Carleson's formula for the Dirichlet integral (see Section \ref{sec:prelim}) shows that if $a \in (0,1)$, then
  $S_{a \mu} f \in \mathcal{D}$ with $\|S_{a \mu}f \| \le \|S_\mu f\|$.
  Moreover, as $a \to 0$, the function $S_{a \mu} f$ converges to $f$ pointwise on $\mathbb{D}$ and hence weakly in $\mathcal{D}$. Thus, $f \in \mathcal{M}$.
  Now, $f \in \mathcal{D} \cap A(\mathbb{D})$ is outer and the boundary zero set of $f$ is countable.
  In this setting, a theorem of Hedenmalm and Shields \cite{HS90}, see also \cite{RS94} or \cite[Section 9.6]{EKM+14}, implies that $f$ is cyclic.

  To finish the proof of the lemma, let $\mathcal{M}_a = S_{a \mu} H^2 \cap \mathcal{D}$. By what was just shown,
  the orthogonal projections $P_{\mathcal{M}_a}$ converge to the identity in SOT as $a \searrow 0$. Thus,
  \begin{equation*}
  \gamma(a \mu) = \frac{P_{\mathcal{M}_a} 1}{\|P_{\mathcal{M}_a}1\|} (0) \xrightarrow{ a\to 0} 1. \qedhere
  \end{equation*}
\end{proof}

The argument in the preceding lemma is qualitative. It would be interesting to get quantitative information
on the rate of convergence of $\gamma(a \mu)$ as $a \to 0$.

We are now ready to construct new singular inner factors.

\begin{theorem}
  \label{thm:new_singular_inner_factor}
  There exists a singular measure $\mu = \sum_{n} c_n \delta_{\tau_n}$ on $\mathbb{T}$, with $\tau_n$
  pairwise distinct points on $\mathbb{T}$,
  such that $S_\mu$ is a factor of $\mathcal{D}$ with the following property:
  There does not exist a subset $C \subset \mathbb{N}$ such that
  \begin{enumerate}
    \item $\{\tau_n: n \in C\}$ is contained in a Carleson set, and
    \item $\sum_{n \in \mathbb{N} \setminus C} \frac{1}{\log ( 1+ \frac{1}{c_n})} < \infty$.
  \end{enumerate}
\end{theorem}

\begin{proof}
  Let $\{b_n\}$ be a sequence in $(0,1)$ such that $\sum_n b_n < \infty$, but $\sum_n \frac{1}{\log \frac{1}{b_n}} = \infty$.
  Let $\{t_n\}$ be a sequence in $(0,1)$ strictly decreasing to zero
  such that writing $\lambda_n = e^{i t_n}$, the set $E = \{\lambda_n : n \in \mathbb{N}\} \cup \{1\}$ is a Carleson set.
  Let $\nu = \sum_n b_n \delta_{\lambda_n}$.
  By Lemma \ref{lem:Carleson_gamma_zero}, there exists a sequence $\{a_k\}$ in $(0,1)$ such that
  $\sum_k (1 - \gamma(a_k \nu)) < \infty$
  and $\sum_k a_k < \infty$. (The second condition actually follows from the first.)
  Recursively, choose a sequence $\{\zeta_k \}$ in $\mathbb{T}$ that is dense in $\mathbb{T}$
  and such that the points $\zeta_k \lambda_n$ for $k,n \in \mathbb{N}$ are all distinct.
  (This is possible since for each $k \in \mathbb{N}$, there are only countably
  many points $\zeta \in \mathbb{T}$ such that $\zeta E \cap (\bigcup_{j=1}^{k-1} \zeta_j E) \neq \emptyset$.)
  Finally, define
  \begin{equation*}
    \mu = \sum_{n,k} a_k b_n \delta_{\zeta_k \lambda_n}.
  \end{equation*}
  We claim that $\mu$ fulfills all conditions of the theorem.

  Let $\nu_k = \sum_n b_n \delta_{\zeta_k \lambda_n}$, hence $\mu = \sum_k a_k \nu_k$.
  By rotation invariance, $\gamma(a_k \nu_k) = \gamma(a_k \nu)$ for all $k$.
  Thus, Lemma \ref{lem:infinite_products} shows that $S_\mu$ is a factor of $\mathcal{D}$.

  Now, let $K \subset \mathbb{T}$ be a closed subset
  and assume that there exists a subset $C \subset \mathbb{N} \times \mathbb{N}$ such that
  $\{\lambda_n \zeta_k: (n,k) \in C\} \subset K$  and $\sum_{(n,k) \in (\mathbb{N} \times \mathbb{N}) \setminus C} \frac{1}{\log \frac{1}{a_k b_n}} < \infty$.
  We will show that $K = \mathbb{T}$; in particular, $K$ cannot be a Carleson set.
  For $k \in \mathbb{N}$, let $C_k = \{n \in \mathbb{N}: (n,k) \in C\}$.
  We claim that each $C_k$ is infinite.
  Indeed, if $C_k$ were finite for some $k$, then there would exist $N \in \mathbb{N}$
  such that for all $n \ge N$, both $n \in \mathbb{N} \setminus C_k$ and $b_n \le a_k$,
  hence
  \begin{equation*}
    \infty > \sum_{(n,k) \in (\mathbb{N} \times \mathbb{N}) \setminus C} \frac{1}{\log \frac{1}{a_k b_n}}
    \ge \sum_{n = N}^\infty \frac{1}{\log \frac{1}{a_k b_n}}
    \ge \frac{1}{2} \sum_{n=N}^\infty \frac{1}{\log \frac{1}{b_n}},
  \end{equation*}
  contradicting our choice of $b_n$. This shows that each $C_k$ is infinite.

  Therefore, for each $k \in \mathbb{N}$, there exists a sequence $\{n_j\}$ in $C_k$ tending to infinity,
  thus $\lambda_{n_j} \zeta_k \in K$ for all $j$. Since $\{\lambda_{n_j}\}$ tends to $1$ and since $K$ is closed,
  it follows that $\zeta_k \in K$ for all $k$. Our choice of $\{\zeta_k\}$ then implies that $K = \mathbb{T}$,
  as asserted.
\end{proof}

As a consequence, we also obtain a zero set that cannot be obtained by combining the Shapiro--Shields theorem
and the Carleson set condition.

\begin{corollary}
  There exists a sequence $\{z_n\}$ in $\mathbb{D} \setminus \{0\}$ that is a zero set of $\mathcal{D}$ such that there does not exist a subset
  $C \subset \mathbb{N}$ such that
  \begin{enumerate}
    \item $\{ \frac{z_n}{|z_n|}: n \in C \}$ is contained in a Carleson set, and
    \item $\sum_{n \in \mathbb{N} \setminus C} \frac{1}{\log(\frac{1}{1 -|z_n|})} < \infty$.
  \end{enumerate}
\end{corollary}

\begin{proof}
  Let $\mu = \sum_n c_n \delta_{\tau_n}$ be a measure provided by Theorem  \ref{thm:new_singular_inner_factor}. Since $c_n \to 0$, we may assume that $c_n \in (0,\frac{1}{2})$ for all $n \in \mathbb{N}$. Let $z_n = (1 -c_n) \tau_n$. Since $S_\mu$ is a factor, Corollary \ref{cor:boundary_implies_zero} shows that $\{z_n\}$ is a zero set, and the remaining properties hold by Theorem \ref{thm:new_singular_inner_factor}.
\end{proof}

\section{Standard weighted Dirichlet spaces}

For $\wedi \in (0,1)$, let $\mathcal{D}_s$ be the space of all holomorphic functions $f$ on $\mathbb{D}$ such that
\begin{equation*}
  \|f\|_\wedi^2 := \|f\|_{H^2}^2 + \int_{\mathbb{D}} |f'(z)|^2 (1 - |z|^2)^\wedi d A(z) < \infty.
\end{equation*}
These spaces are weighted Dirichlet spaces.
In this section, we will extend the sufficiency part of Theorem \ref{thm:SS_intro} to the spaces $\mathcal{D}_\wedi$.
The key is the following partial analogue of Corollary \ref{cor:atomic_gamma}. Since the proof is quite different, we provide the details. We will write $\gamma_\wedi = \gamma_{\mathcal{D}_\wedi}$ for the value of the extremal function at the origin.

\begin{lemma}
  \label{lem:atomic_gamma_wedi}
  Let $\wedi \in (0,1), \lambda \in \mathbb{T}, a > 0$ and let $S_a(z) = \exp(-a \frac{\lambda+z}{\lambda-z})$. Then
  \begin{equation*}
    1 - \gamma_\wedi(S_a H^2 \cap \mathcal{D}_\wedi) \le c a^\wedi \quad \text{ as } a \to 0,
  \end{equation*}
  where $c> 0$ is a constant that only depends on $\wedi$.
\end{lemma}

\begin{proof}
  By rotation invariance, we may assume that $\lambda = 1$.
  We use the function
  \begin{equation*}
    f(z) = S_a(z) \frac{1-z}{1- r z},
  \end{equation*}
  where $r \in (0,1)$ will be chosen later. Since
  \begin{equation*}
    f'(z) = - \frac{2 a}{(1-z)^2} S_a(z) \frac{1-z}{1-r z} + S_a(z) \frac{r - 1}{(1 - r z)^2},
  \end{equation*}
  we find that
  \begin{equation*}
    |f'(z)|^2 \lesssim \frac{a^2}{|1-z|^2} \frac{1}{|1 - r z|^2} + \frac{(1-r)^2}{|1 - r z|^4}.
  \end{equation*}
  Known integral estimates   show that
  \begin{equation*}
    \int_{\mathbb{D}} \frac{(1 - |z|^2)^\wedi}{|1 - r z|^4} d A(z)
    \lesssim
    \int_{\mathbb{D}} \frac{(1 - |z|^2)^\wedi}{|1 - r z|^2 |1-z|^2} d A(z)
\lesssim \frac{1}{(1-r)^{2 - \wedi}}.
  \end{equation*}
  A reference for this, which also applies to the unit ball of $\C^n$, is  \cite[Lemma 2.5]{OF96_1}.
  In fact, that the left hand side is bounded by the right hand side is what is often referred to as a Forelli--Rudin estimate (for the one dimensional case also see \cite{SW71}, Lemmas 5 and 6). That the term in the middle is equivalent in size to the term on the right can be derived by splitting the integration into the two cases $|1-z|<(1-r)/2$ and $|1-z|\ge (1-r)/2$. In the first region one can use the change of variables $w=1-z$ along with some straightforward inequalities, and in the second region one has $|1-rz|\lesssim |1-z|$ and can use the standard estimate.
  Thus,
  \begin{equation*}
    \int_{\mathbb{D}} |f'(z)|^2 (1 - |z|^2)^\wedi d A(z) \lesssim \frac{a^2}{(1 - r)^{2 - \wedi}} + (1 - r)^\wedi.
  \end{equation*}
  We now choose $r= 1 - a$, which makes the last quantity comparable to $a^s$.
  Since
  \begin{equation*}
    \|f\|_{H^2}^2 \le \|f\|_{H^\infty}^2 = \frac{4}{(1 +r)^2} = \frac{4}{(2 -a)^2} = 1 + O(a),
  \end{equation*}
  it follows that
  \begin{equation*}
    \|f\|_{\wedi}^2 = \|f\|_{H^2}^2 + \int_{\mathbb{D}} |f'(z)|^2 (1 - |z|^2)^\wedi d A(z)
    \le 1 + c a^\wedi
  \end{equation*}
  for some $c > 0$. Hence
  \begin{equation*}
    \gamma_{\mathcal{D}_\wedi}(S_a H^2 \cap \mathcal{D}_\wedi) \ge \frac{f(0)}{\|f\|_{\mathcal{D}_\wedi}} \ge \frac{e^{-a}}{\sqrt{1 + c a^\wedi}} = 1 - \frac{c}{2} a^\wedi + o(a^{\wedi}),
  \end{equation*}
  as desired.
\end{proof}

\begin{remark}
  The method of proof in Lemma \ref{lem:atomic_gamma_wedi} can also be used to obtain the upper bound in Corollary \ref{cor:atomic_gamma} in the Dirichlet space. However, it is not sufficient to work with $f(z) =  S_a(z) \frac{1 - z}{1 - r z}$, since
  \begin{equation*}
    \|f\|_{\mathcal{D}}^2 \ge \Big\| \frac{1 - z}{1 - r z} \Big\|_{\mathcal{D}}^2 \ge \frac{5}{4},
  \end{equation*}
  where the last inequality follows from an elementary computation with power series coefficients or from the observation that the image of the map $z \mapsto \frac{1-z}{1 - r z}$ is a disc of radius $\frac{1}{1+r} \ge \frac{1}{2}$. Thus, $\|f\|_{\mathcal{D}}$ does not tend to $1$ as $a \to 0$, so $f$ is not useful for showing that $\gamma_\mathcal{D}(S_a H^2 \cap \mathcal{D}) \to 1$ as $a \to 0$.

  Instead, one can replace $f$ with $g(z) = S_a(z) (1 - z) p_n(z)$, where $p_n$ is the degree $n$ polynomial minimizing
  $\|(1- z) p_n - 1\|_{\mathcal{D}}$. Such polynomials are called optimal polynomial approximants and were computed in \cite[Lemma 3.1]{BCL+15}. Using those results and choosing $n = \lceil a^{-1/2} \rceil$, one can recover the bound $1 - \gamma_\mathcal{D}(S_a H^2 \cap \mathcal{D}) \lesssim \frac{1}{\log \frac{1}{a}}$ of Corollary \ref{cor:atomic_gamma}. We omit the details.
\end{remark}

With Lemma \ref{lem:atomic_gamma_wedi} in hand, we can now prove an analogue of the sufficiency part of Theorem \ref{thm:SS_intro}.

\begin{theorem}
  \label{thm:SS_wedi}
  Let $\wedi \in (0,1)$.
  Let $\{a_n\}$ be a sequence of positive numbers such that
  \begin{equation*}
    \sum_n a_n^s < \infty
  \end{equation*}
  and let $\{\lambda_n\}$ be a sequence in $\mathbb{T}$. Let $S_\mu$ be the singular inner function corresponding to
  \begin{equation*}
    \mu = \sum_n a_n \delta_{\lambda_n}.
  \end{equation*}
  Then there exists $g \in H^2 \setminus \{0\}$ such that $S_\mu g \in \mathcal{D}_\wedi$.
\end{theorem}

\begin{proof}
  As in the proof of Proposition \ref{prop:SS_singular_suff} and Lemma \ref{lem:infinite_products}, we can take an infinite product of extremal functions.
  The key point is that extremal functions continue to be contractive multipliers of $\mathcal{D}_\wedi$, since
  $\mathcal{D}_\wedi$ is a complete Pick space by a theorem of Shimorin \cite{Shimorin02}, see also the comments in Section \ref{sec:prelim}. Lemma \ref{lem:atomic_gamma_wedi} and the hypothesis on $\{a_n\}$ then guarantee convergence of the infinite product.
\end{proof}

In the setting of the last theorem, the function $g$ necessarily belongs to $\mathcal{D}_s$;
see \cite[Theorem IV.3.4]{Aleman93}.

\appendix
\section{A radial complete Pick space where points cannot be pushed out}
\label{sec:Appendix A}

In this appendix, we construct a radial complete Pick space in which the analogue of Theorem \ref{thm:pushing_out} fails.
It is well-known that extremal functions for finite zero sets can be represented by use of the determinants of certain Gramians associated with the reproducing kernel functions, see \cite{SS62}. For numerical experiments we found the following variation useful. The argument is standard, for similar results, see e.g.\ \cite[Theorem 3.4]{PR16}.

\begin{lemma}
  \label{lem:compute_gamma}
  Let $\mathcal{H}$ be an RKHS on $X$ with kernel $K$ that is normalized at a point $z_0 \in X$, meaning
  that $K(z,z_0) = 1$ for all $z \in X$.
  Let $\{z_1,\ldots,z_n\} \subset X$ be a set of $n$ distinct points. Let
  $L = [K(z_i,z_j)] \in M_n(\mathbb{C})$, and let $\ones \in \mathbb{C}^n$ be the all-ones vector.

  Let
  \begin{equation*}
    \gamma = \sup\{ |f(z_0)|: f(z_k) = 0 \text{ for } 1 \le k \le n, \|f\| \le 1\}.
  \end{equation*}

  If $M\in M_n(\mathbb{C})$ is any matrix that satisfies $LML=L$, then
  \begin{equation*}
    \gamma^2 = 1-\langle M \ones, \ones \rangle .
  \end{equation*}
\end{lemma}
Thus, if $L$ is invertible, then we can take $M=L^{-1}$. If $L$ is not invertible, which happens when the kernel functions $K_{z_1}, \dots, K_{z_n}$ are linearly dependent, then one can take $M=L^+$, the Moore-Penrose inverse of $L$.

\begin{proof}  Let $\mathcal N$ be the linear span of  $K_{z_1}, \dots, K_{z_n}$. We will write $P_\mathcal N$ and $P_{\mathcal N^\perp}$ for the projections onto $\mathcal N$ and $\mathcal N^\perp$.

Then since $K_{z_0}=1$ we have $\gamma=\sup\{|\langle f, 1\rangle|: f\in \mathcal N^\perp, \|f\|\le 1\}$ and
for   $f\in \mathcal N^\perp$  $$|f(z_0)|=|\langle f, 1\rangle| = |\langle f, P_{\mathcal N^\perp} 1\rangle |\le  \|f\|\|P_{\mathcal N^\perp} 1\|,$$ and equality holds, if and only if $f$ is a multiple of $P_{\mathcal N^\perp} 1$. Thus, $\gamma^2=\|P_{\mathcal N^\perp} 1\|^2=1-\|P_{\mathcal N} 1\|^2$.

Let
  \[
    T: \mathcal{N} \to \mathbb{C}^n, \quad f \mapsto ( \langle f,K_{z_j} \rangle)_{j=1}^n.
  \]
  Then $T P_{\mathcal{N}}1 = \ones$.
  One checks that $T^* e_j = K_{z_j}$ for each $j$. In particular, $T^*$ is surjective, so
  $T$ is injective, and we have $L = T T^*$. Thus, if a matrix $M$ satisfies $LML = L$,
  then $T T^* M T T^* = T T^*$ and hence $T^* M T = \operatorname{id}_{\mathcal{N}}$. It follows that
  \begin{equation*}
    1 - \gamma^2 = \|P_{\mathcal{N}} 1 \|^2 = \langle P_{\mathcal{N}} 1, P_\mathcal{N} 1 \rangle
    = \langle T^* M T P_{\mathcal{N}}1 , P_{\mathcal{N}} 1 \rangle
    = \langle M \ones, \ones \rangle. \qedhere
  \end{equation*}
\end{proof}

If $n$ is large, then the expression for $\gamma$ in Lemma \ref{lem:compute_gamma} may only be useful for numerical experiments. However, for two points one obtains a simple formula that we will take advantage of in a moment.
 If $K$ is any reproducing kernel on $X$, and if $r,t\in X$ such that $K(\cdot,r)$ and $K(\cdot,t)$ are linearly independent, then the $2\times 2$ Gramian
$$L= \left[\begin{matrix} K(r,r) &K(t,r)\\ K(r,t)&K(t,t)\end{matrix}\right]$$ has inverse
$$ L^{-1}=\frac{1}{\det L} \left[\begin{matrix} K(t,t) &-K(t,r)\\ -K(r,t)&K(r,r)\end{matrix}\right].$$
Thus,
\begin{align}\label{equ:2pts}\la L^{-1}\ones, \ones\ra= \frac{K(r,r)+K(t,t)-2 \mathrm{Re}K(t,r)}{K(r,r)K(t,t)-|K(t,r)|^2}.\end{align}

\

We now come to the announced example.
The basic idea is to construct a space that privileges even functions over odd functions, so that zeros tend to come in pairs $\pm z$, and pushing out only one of those zeros decreases $\gamma$.
\begin{example}
  Let $a \in [0,1)$ and let $\mathcal{H}_a$ be the RKHS on $\mathbb{D}$ with kernel
  \begin{equation*}
  K_a(z,w) = \frac{1}{1 - a (z \overline{w}) - (1 - a) (z \overline{w})^2}.
  \end{equation*}
  Since $1 - \frac{1}{K_a}$ is positive semi-definite, each $\mathcal{H}_a$ is a complete Pick space, normalized at $0$.
  Fix $r \in (0,1)$. For $t \in (0,1)$, let
  \begin{equation*}
    \gamma_a(t) = \sup \{ |f(0)|: f(-r) = f(t) = 0, \|f\|_{\mathcal{H}_a} \le 1 \}.
  \end{equation*}
  We claim that for sufficiently small $a \ge 0$, the function $\gamma_a$ is not increasing.

  If $a = 0$, then every function in $\mathcal{H}_a$ is even, which will imply
  that $\gamma_0$ has a local maximum at $t=r$.
  Indeed, the form of the kernel implies that if $\varphi: \mathbb{D} \to \mathbb{D}, \varphi(z) = z^2$,
  then
  \begin{equation*}
    H^2 \to \mathcal{H}_0, \quad g \mapsto g \circ \varphi,
  \end{equation*}
  is unitary.
  From this observation, we conclude that
  \begin{equation*}
    \gamma_0(t) = \sup \{ |g(0)|: g(r^2) = g(t^2) = 0, \|g\|_{H^2}\ \le 1\}.
  \end{equation*}
  If $t \neq r$, then the solution of the extremal problem in $H^2$ is given by the Blaschke product $\frac{(r^2 - z)(t^2 - z)}{(1 - r^2 z)(1 - t^2 z)}$,
  so $\gamma_0(t) = r^2 t^2$.
  If $t = r$, then the solution of the extremal problem in $H^2$ is the Blaschke product $\frac{r^2 - z}{1 - r^2 z}$,
  so $\gamma_0(t) = r^2$. It follows that $\gamma_0$ is the discontinuous function given by
  \begin{equation*}
    \gamma_0(t) =
    \begin{cases}
      r^2 t^2 & \text{ if } t \neq r \\
      r^2 & \text{ if } t  = r.
    \end{cases}
  \end{equation*}
  Notice that for $t \neq r$, the extremal function in $\mathcal{H}_0$ has four zeros, namely at $\pm t, \pm r$, whereas
  for $t = r$, the extremal function only has the two zeros $\pm r$.

  By choosing small $a > 0$, we also obtain examples where $\mathcal{H}_a$ contains
  all polynomials, and hence $K_a$ has full rank.
  In principle,  Lemma \ref{lem:compute_gamma} together with formula (\ref{equ:2pts})
  yields an explicit (albeit somewhat complicated) formula for $\gamma_a(t)$, which is
  the square root of a rational function in $t$.
  This formula can be used to show that $\gamma_a$ is not increasing for small $a > 0$. For instance,
  Figure \ref{fig:gamma_plot} shows a plot of $\gamma_a$ for $a = 0.01$ and $r = 0.5$.

  \begin{figure}[ht]

\begin{tikzpicture}
  \begin{axis}[
      xlabel={$t$},
      domain=0:1,
      samples=200,
      grid=both,
  ]
  \addplot[mark=none, line width = 1.2pt] {x * sqrt(3881196*x^4 - 4038012*x^3 + 1089495*x^2 - 40788*x + 158812) / (40*sqrt(39501*x^2 - 39600*x + 10300))};
  \end{axis}
\end{tikzpicture}
    \caption{Plot of $\gamma_a(t)$ for $a=0.01$ and $r = 0.5$}
    \label{fig:gamma_plot}
  \end{figure}
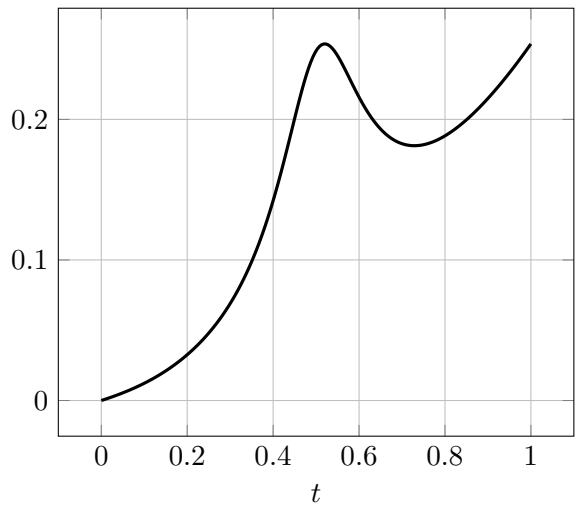

  Alternatively, one can show that $\gamma_a$ is not increasing for sufficiently small $a> 0$
  by showing that $\lim_{a \searrow 0} \gamma_a(t) = \gamma_0(t)$ for all $t \in (0,1)$.
  If $t \neq r$, then the $2 \times 2$ matrices
  \begin{equation*}
    L_a(t) =
    \begin{bmatrix}
      K_a(-r,-r) & K_a(-r,t) \\
      K_a(t,-r) & K_a(t,t)
    \end{bmatrix}
  \end{equation*}
  are invertible for all $a \ge 0$, since each $\mathcal{H}_a$ separates the points $t$ and $-r$, so  Lemma \ref{lem:compute_gamma} and formula (\ref{equ:2pts}) show that $\lim_{a \searrow 0} \gamma_a(t) = \gamma_0(t)$.

  Next, let $t = r$. We have already remarked that $\gamma_0(r)=r^2$. For $a>0$ the kernels $K_a(\cdot,r)$ and $K_a(\cdot,-r)$ are linearly independent, so the matrix $L_a(r)$ is invertible. Thus, we may use Lemma \ref{lem:compute_gamma} and formula (\ref{equ:2pts}) with $(r,t)=(r,-r)$.
  Define $F(a)=1-(1-a)r^4$ and $G(a)= ar^2$. Then $K_a(r,r)=K_a(-r,-r)=\frac{1}{F(a)-G(a)}$ and $K_a(r,-r)=K_a(-r,r)=\frac{1}{F(a)+G(a)}$. Hence after a short calculation  we obtain
$$1-\gamma_a(r)^2=\frac{F^2(a)-G^2(a)}{F(a)}\to F(0)= 1-r^4 \ \text{ as } a\searrow 0.$$ Thus, $\gamma_a(r)\to r^2=\gamma_0(r)$ as $ a \searrow  0.$
\end{example}

\section{Added during revisions: The answer to Question \ref{ques:SingVsZeros}}
\label{sec:appendixb}

During the revision phase of this manuscript, an affirmative resolution to Question \ref{ques:SingVsZeros} was found by ChatGPT-6 Astra.
Below, we present a streamlined and rewritten version of this proof.
We also note that we were subsequently informed by Alberto Dayan and Artur Nicolau that they have independently obtained a solution to Question \ref{ques:SingVsZeros}, similarly utilizing AI assistance.

\begin{theorem} For $k\in \N$ let $r_k\in (0,1)$ and $\lambda_k\in \T$ such that $\{r_k\lambda_k\}\subseteq \D$ is a Blaschke sequence. Define $\mu= \sum_k (1-r_k)\delta_{\lambda_k}$. 

Then $\{r_k\lambda_k\}\subseteq \D$ is a zero set for $\mathcal D$ if and only if $S_\mu$ is a factor of $\mathcal D$.
\end{theorem}
\begin{proof} In light of Corollary \ref{cor:boundary_implies_zero} we only have to show that if $\{r_k\lambda_k\}\subseteq \D$ is a zero set for $\mathcal D$, then $S_\mu$ is a factor of $\mathcal D$. Write $b_k(z)= \frac{r_k-\overline{\lambda_k}z}{1-r_k\overline{\lambda_k}z}$. Then $B=\prod_{k=1}^\infty b_k$ is the Blaschke product with zeros $\{r_k\lambda_k\}$ and $B(0)>0$. The hypothesis implies that there is a non-zero $h\in \mathcal D$ such that $Bh\in \mathcal D$. 

For each $k\in \N$ define $q_k=\frac{1+b_k}{2}$. Then $|q_k(z)|\le 1$ for all $z\in \D$ and $b_k(0)\le q_k(0)$. Hence
\begin{equation*}
  \sum_k |1 - q_k(0)| \le \sum_k |1 - b_k(0)|  <\infty.
\end{equation*}
Thus, as in the proof of Lemma \ref{lem:infinite_products} it follows that
the infinite product $Q(z)=\prod_{k=1}^\infty q_k(z)$ converges locally uniformly to a nonzero analytic function of $H^\infty$-norm $\le 1$. We will establish the theorem by showing that $Qh$ and $S_\mu Qh\in \mathcal D$.
(Note that $q_k(\lambda_k) = 0$, which helps cancel the singularity of $S_\mu$ at $\lambda_k$.)

For $N\in \N$ write $B_N=\prod_{k=1}^N b_k$ and $Q_N=\prod_{k=1}^N q_k$.
For each $k$, the triangle inequality and equation (\ref{equ:D_z(I)}) imply that
\begin{equation*}
  \|q_k f\| \le \frac{1}{2} ( \|f\| + \|b_k f\|) \le \|b_k f\|.
\end{equation*}
Inductively, it follows that
\begin{equation*}
  \|Q_N h\| \le \|B_N h\| \le \|B h\| < \infty,
\end{equation*}
where we again used \eqref{equ:D_z(I)} in the second estimate. Using the Fatou property of $\mathcal{D}$,
we conclude that $Q h \in \mathcal{D}$ and $\|Q h\| \le \|B h\|$. (This estimate could be slightly strengthened by using
the parallelogram law rather than the triangle inequality, but we will not do so here.)

Next we note that for each $k\in \N$ we have $$|Q(z)|\le |q_k(z)|= \frac{1+r_k}{2}\left|\frac{1-\overline{\lambda_k}z}{1-r_k\overline{\lambda_k}z}\right|,$$
 initially for $z \in \mathbb{D}$ and then also for a.e.\ $z \in \mathbb{T}$.
Hence 
 for a.e.\ $z\in \T$ we obtain from (\ref{eqn:local_Diri_inner})
 \begin{align*}D_z(S_\mu)|Q(z)|^2&=  \sum_k \frac{2(1-r_k)}{|z-\lambda_k|^2}\ |Q(z)|^2\\
 &\le \sum_k \frac{1+r_k}{2} \frac{1-r^2_k}{|1-r_k\overline{\lambda_k}z|^2}\\&\le D_z(B),\end{align*}
This implies
\begin{align*}\|S_\mu Qh\|^2 &=\|Qh\|^2 + \int_{\T} D_z(S_\mu)|Qh|^2\frac{|dz|}{2\pi}\\
&\le \|Bh\|^2+ \int_{\T} D_z(B)|h|^2\frac{|dz|}{2\pi}\\& \le 2\|Bh\|^2<\infty.
\qedhere
\end{align*}
\end{proof}
\bibliographystyle{amsplain}
\bibliography{literature}

\end{document}